\documentclass[a4paper,11pt]{article}
\usepackage[latin1]{inputenc}
\usepackage{german}
\usepackage{umlaute}
\usepackage{amssymb}
\usepackage{calc}
\usepackage{times}
\usepackage{amsmath,amstext,amsthm,amsfonts, ams fonts}
\usepackage{fancyhdr}
\usepackage{a4wide}
\usepackage{wasysym}
\usepackage[scaled]{helvet}
\usepackage{latexsym}
\usepackage{rotate}
\usepackage[normalem]{ulem}
\usepackage{remreset}
\usepackage{float}
\makeatletter
\@removefromreset{footnote}{chapter}
\makeatother
\usepackage{bbm}
\usepackage[nottoc]{tocbibind}
\usepackage{lscape}
\usepackage{caption2}
\usepackage{longtable}
\usepackage{rotating}
\usepackage{dsfont}
\usepackage{pifont}
\usepackage{nicefrac}
\usepackage{colortbl,color,graphics,graphicx}
\usepackage{sublabel}
\usepackage{fancyhdr}
\usepackage{eucal}
\usepackage{longtable}
\usepackage{natbib}
\usepackage[clearempty]{titlesec}
\usepackage{graphics,graphicx,color,epsfig}
\usepackage[all]{xy}
\usepackage{titlesec}
\usepackage{makeidx}

\renewcommand{\d}{\Delta}
\renewcommand{\P}{\mathbb{P}}
\newcommand{\E}{\mathbb{E}}

\newcommand{\R}{\mathds{R}}

\newcommand{\KLEINO}{{\scriptstyle{\mathcal{O}}}}
\newcommand{\txg}{\tilde X_{g_i}}

\newcommand{\txl}{\tilde X_{l_{i-K}}}

\newcommand{\tyg}{\tilde Y_{\gamma_i}}

\newcommand{\tyl}{\tilde Y_{\lambda_{i-K}}}
\newcommand{\F}{\mathcal{F}}

\newcommand{\ec}{\textcircled{e}}
\newcommand{\nc}{\textcircled{n}}
\newcommand{\mc}{\textcircled{m}}
\newcommand{\mic}{\textcircled{$\nu$}}
\renewcommand{\1}{\mathbbm{1}}
\def\lsim{~\rlap{$<$}{\lower 1.0ex\hbox{$\sim$}}}

\DeclareSymbolFont{largesymbols}{OMX}{yhex}{m}{n}
\DeclareMathAccent{\verywidehat}{\mathord}{largesymbols}{'144}

\long\def\symbolfootnote[#1]#2{\begingroup%
\def\thefootnote{\fnsymbol{footnote}}\footnote[#1]{#2}\endgroup}
\long\def\symbolfootnotetext[#1]#2{\begingroup%
\def\thefootnote{\fnsymbol{footnote}}\footnotetext[#1]{#2}\endgroup}
\def\lsim{\mathrel{\rlap{\lower4pt\hbox{\hskip1pt$\sim$}}
    \raise1pt\hbox{$\le$}}}
\newcommand{\var}{\mathbb{V}\hspace*{-0.05cm}\textnormal{a\hspace*{0.02cm}r}}

\newcommand{\cov}{\mathbb{C}\textnormal{o\hspace*{0.02cm}v}}
\renewcommand{\:}{\mathrel{\mathop{:}}}
\newtheorem{defi}{Definition}
\newtheorem{remark}[defi]{Remark}
\newtheorem{prop}[defi]{Proposition}
\newtheorem{lem}[defi]{Lemma}
\newtheorem{theo}{Theorem}
\newtheorem{cor}[defi]{Corollary}
\newtheorem{annahme}{Assumption}

\renewcommand{\figurename}{Figure}
\begin{document}
\pagestyle{plain}
\title{Efficient Covariance Estimation for Asynchronous Noisy High-Frequency Data}
\begin{titlepage}
\begin{center}
\Large
\textbf{Efficient Covariance Estimation for Asynchronous Noisy High-Frequency Data}\\ \vspace*{1.25cm}
\begin{center}\large  Markus Bibinger \footnote[1]{E-mail: bibinger@math.hu-berlin.de}\\ \vspace*{0.5cm}
Interdisciplinary Center for Scientific Computing\\ of the Ruprecht-Karls-University of Heidelberg\\
and\\Department of Mathematics\\
Humboldt University of Berlin\\
\end{center}
 \large December 2008 \\ \vspace*{3.5cm}
 \normalsize\textbf{Abstract}\end{center}
We focus on estimating the integrated covariance of log-price processes in the presence of market microstructure noise. 
We construct an efficient unbiased estimator for the quadratic covariation of two It\^{o} processes in the case where high-frequency asynchronous discrete returns under market microstructure noise are observed. This estimator is based on synchronization and multi-scale methods and attains the optimal rate of convergence. A Monte Carlo study analyzes the finite sample size characteristics of our estimator. 
\noindent
\textbf{MSC classes:} 62F12, 62G05
\end{titlepage}
\setcounter{page}{2}
\section{Introduction\label{sec:1}}
Estimating the quadratic covariation, also called integrated covariance, of asset returns is a central theme in finance. With the availability of high-frequency intraday returns the estimation of daily integrated covariances using high-frequency observations became an issue of great interest. The problems occurring in covariance estimation using high-frequency data are mainly the lack of synchronicity and market microstructure noise.\\
In this article we propose a new estimator for the integrated covariance $\langle \tilde X,\tilde Y\rangle_T$ of two log-price processes over a fixed time horizon $[0,T]$ (usually one trading day) when we observe high-frequent noisy asynchronous data. The problem of asynchronous data without noise was solved by \cite{hy} and there are as well estimators developed in recent literature that solve the problem of noisy but synchronous data (see e.\,g.\,\cite{barndorff8}).\\
We work within the model where the efficient asset processes $\tilde X$ and $\tilde Y$ (without noise) are assumed to be It\^{o} processes
\begin{align*}d\tilde X_t&=\mu_t^X~dt+\sigma_t^X~dB_t^X~,\\
d\tilde Y_t&=\mu_t^Y~dt+\sigma_t^Y~dB_t^Y~,~~~t \in [0,T]\end{align*}
with Brownian motions $B^X$ and $B^Y$ which are correlated with $corr(B^X_t,B^Y_t)=\rho_t$ and continuous, bounded and adapted stochastic processes $\mu_t^X,\mu_t^Y,\sigma_t^X,\sigma_t^Y$. It is well known that for synchronous observations without noise the realized covariance $\sum_{t_{i+1}\le T}\left(X_{t_{i+1}}-X_{t_i}\right)\left(Y_{t_{i+1}}-Y_{t_i}\right)$ is a consistent estimator for $\langle \tilde X,\tilde Y\rangle_T=\int_0^T\rho_t\sigma_t^X\sigma_t^Y~dt$ if $\sup_i{\left(t_{i+1}-t_i\right)}\rightarrow 0$.\\ In the case of high-frequency data, e.\,g.\,tick-by-tick data, the observations of the asset processes are usually not simultaneous and estimation methods are hence based on synchronization of the data by interpolation (e.\,g.\,linear or previous-tick interpolation) before calculating a realized covariance estimator.\\ We use the construction of a synchronized time grid for two processes following the method proposed by \cite{palandri}. The realized covariance calculated with the synchronized observations corresponds to the Hayashi-Yoshida estimator given by the sum of all products of increments with overlapping time intervals. This estimator is proved to be a consistent estimator in the absence of noise (\cite{hy}) but becomes inconsistent when market microstructure effects are relevant. The behaviour of this estimator and the realized covariance depending on the sample frequencies are studied in \cite{voev} and \cite{oomen}.\\
Market micostructure effects and estimators for the integrated volatility under its influence were studied intensively in recent literature. Estimators with an optimal rate of convergence $N^{\nicefrac{1}{4}}$, where $N$ denotes the number of observations, are presented by \cite{zhang} and \cite{bn}. Merging those techniques for asynchronous data and the subsampling approach to high-frequency observations contaminated by market microstructure noise as presented in \cite{zhangmykland} a consistent estimator for asynchronous noisy data can be achieved. By construction of an adequate synchronized time-scale for two asset processes and subsampling an estimator with $N^{\nicefrac{1}{6}}$-rate of convergence can be obtained, where $N$ denotes the number of synchronized observations in this context, which is less than or equal to the minimum of observations of $X$ and $Y$. This result is presented in \cite{palandri}. We use the same methods to rearrange the observations in a synchronized grid and show how an extension of subsampling to a multi-scale approach can afford an estimator with a more efficient rate of convergence $N^{\nicefrac{1}{4}}$, which we prove to be the best attainable rate. For this purpose we prove local asymptotic normality with rate $N^{-\nicefrac{1}{4}}$ for a simplified model and obtain bounds for the asymptotic Fisher information. With the minimax theorem we conclude that $N^{\nicefrac{1}{4}}$ is a lower bound for the rate of convergence even in the synchronous equidistant case. Our estimator hence upgrades the estimator proposed in \cite{palandri} to a rate-optimal consistent estimator for integrated covariances and leads to a suitable implementation of integrated variance and covariance estimation following the multi-scale approach invented by \cite{zhang} and our extension for the covariance case. \cite{barndorff9} present a multivariate realized kernel estimator that, furthermore, guarantees to be positive semi-definite, which is aside from non-synchronicity and noise a third important issue in multivariate considerations. Their estimator has a $N^{\nicefrac{1}{5}}$-rate of convergence.\\ In Section \ref{sec:2} we introduce the model and our basic notation. We present a brief outline and the two main results of this article concerning the asymptotics of our multi-scale estimator for the quadratic covariation and local asymptotic normality.  

\section{Model and Main Results\label{sec:2}}
We want to obtain a consistent estimator for the covariation $\langle \tilde X, \tilde Y\rangle_T$ of two It\^{o} processes $\tilde X_t$ and $\tilde Y_t$ over a fixed time interval $[0,T]$, e.\,g.\,one trading day, when we observe discrete asynchronous returns contaminated by market microstructure noise. 
The observations of $X$ will be denoted by 
$$X_{t_0},\ldots ,X_{t_n}\,,~~~~~~0\le t_0<t_1<\ldots<t_n\le T\mbox{, with increments}~ \d X_{t_i}=X_{t_{i}}-X_{t_{i-1}}~,$$
and the observations of another log-price process $Y$ by 
$$Y_{\tau_0},\ldots,Y_{\tau_m}\,,~~~~~~0\le \tau_0<\tau_1<\ldots<\tau_m\le T\mbox{, with increments}~ \d Y_{\tau_j}=Y_{\tau_{j}}-Y_{\tau_{j-1}}~.$$
Synchronous data would mean that $m=n$ and $t_i=\tau_i$ for all $i \in \{0,\ldots,n\}$. We consider the general case where the number of observations may differ and the sets of observation times $\mathcal{O^X}=\{t_0,\ldots,t_n\}$ and $\mathcal{O^Y}=\{\tau_0,\ldots,\tau_m\}$  also contain points $t_i \notin \mathcal{O^Y}$ and $\tau_j \notin \mathcal{O^X}$. Usually the considered time interval $[0,T]$ starts with the first observation which means $t_0=0$ or $\tau_0=0$.
We work within the model imposed by the following two assumptions: 
\begin{annahme}\label{eff}
The observed log-price processes are described by the sums of efficient It\^{o} processes and independent (discrete-time) noise processes $\epsilon^X_{t_{i}}$ and $\epsilon^Y_{\tau_{j}}$:
$$X_{t_i}=\tilde X_{t_i}+\epsilon_{t_{i}}^X~,~~i \in \{0,\ldots,n\}~~~\phantom{.}$$
$$Y_{\tau_j}=\tilde Y_{\tau_j}+\epsilon_{\tau_{j}}^Y~,~~j \in \{0,\ldots,m\}~.$$
On a filtered probability space $(\Omega,\F,(\F_t))$ the efficient processes are defined by
\begin{align*}d\tilde X_t&=\mu_t^{X}dt+\sigma_t^XdB_t^X~, \\
d \tilde Y_{\tau}&=\mu_{\tau}^Y d\tau+\sigma_{\tau}^YdB_{\tau}^Y~,\end{align*}
where $B^X$ and $B^Y$ are two $\F_t$-adapted correlated standard Brownian motions with correlation $\rho_t$. The drift $\mu_t$ and spot volatility $\sigma_t$ for both are $\F_t$-adapted, continuous and bounded stochastic processes.
\end{annahme}
\begin{annahme}\label{e} The errors $\epsilon_{t_i}^X,i \in \{0,\ldots,n\}$ and $\epsilon_{\tau_j}^Y,j\in\{0,\ldots,m\}$ due to market microstructure noise are assumed to be i.i.d.\,processes and independent to each other and the efficient processes. We also assume $$\E\epsilon^X\:=\E\epsilon_{t_i}^X=\E\varepsilon^Y\:=\E\varepsilon_{\tau_j}^Y=0~~~~\forall i,j$$ and $\E\left(\epsilon^X\right)^4,\E\left(\epsilon^Y\right)^4<\infty$ where $\E\left(\epsilon^X\right)^k\:=\E\left[\left(\epsilon_{t_i}^X\right)^k\right]$ and $\E\left(\epsilon^Y\right)^k\:=\E\left[\left(\varepsilon_{\tau_j}^Y\right)^k\right]$ analogously.
\end{annahme}
The variances of the noise processes will be denoted by
$$\eta_X^2\:=\E\left[\left(\epsilon_{t_i}^X\right)^2\right]~~\text{and}~~\eta_Y^2\:=\E\left[\left(\epsilon_{\tau_j}^Y\right)^2\right]~.$$
We want to obtain a consistent rate-optimal estimator for the integrated covariance of the efficient processes $\langle \tilde X,\tilde Y\rangle_T=\int_0^T\rho_t\sigma_t^X\sigma_t^Y~dt$ under asymptotics where $\sup_i\d t_i \rightarrow 0$ and $\sup_j \d \tau_j \rightarrow 0$.\\
Assumptions concerning the observation times are imposed later on after constructing a synchronized joint grid. Evidently, we need a form of regularization criterion to ensure the number of observations for both processes and the length of (usually not equidistant) time intervals $\d t_i=\left(t_i-t_{i-1}\right)$ and $\d \tau_j=\left(\tau_j-\tau_{j-1}\right)$ being of the same order. Recall that for our analysis we regard the conditional law given the observation times. A precise analysis for the case of random trading times (e.\,g.\,event times of counting processes) require some additional concepts that are not the focus of this article although we will use Poisson processes to generate the observation times in Section \ref{sec:7}. The latter could be interesting when tick-by-tick data are considered, where trading times can be modeled as Poisson arrivals. We refer to \cite{zhangeps} for an analysis of this special case. \\ \\
At this point we present an outline and an outlook on the the two main results of this article. 
The article is organized as follows: in Section \ref{sec:3} we describe the method of synchronization and show that the Hayashi-Yoshida estimator becomes inconsistent in the presence of market microstructure noise. In Section \ref{sec:4} we present the subsampling estimator, which is similar to the one proposed by \cite{palandri}, but using a different representation, which will be useful for our further analysis. In Section \ref{sec:5} we develop our new estimator using a multi-scale approach that improves the rate of convergence to $N^{\nicefrac{1}{4}}$ resulting in a more efficient estimation method for the quadratic covariation $\langle \tilde X, \tilde Y\rangle_T$. This is one of our main results and can be summarized in the following Theorem \ref{multitheo}:
\begin{theo}\label{multitheo}
Let Assumtions \ref{eff}, \ref{e} and \ref{grid}, that will be stated in Section \ref{sec:3}, be satisfied.\\
We will prove that the multi-scale estimator $\widehat{\langle \tilde X,\tilde Y\rangle}_T^{(mult)}$, which will be constructed in Sections \ref{sec:3}-\ref{sec:5}, is unbiased and has the following asymptotic property:
\begin{align*}\left(\E\left[\left(\widehat{\langle \tilde X,\tilde Y\rangle}_T^{mult}-\langle \tilde X,\tilde Y\rangle_T\right)^2\right]\right)^{\nicefrac{1}{2}}=\mathcal{O}\left(N^{-\nicefrac{1}{4}}\right)~.\end{align*}
$N$ denotes the number of synchronized observations emanating from the synchronization method that will be presented in the next section.\end{theo}
In Section \ref{sec:5}, after the construction of our multi-scale estimator, Proposition  \ref{multi} gives a more detailed version of Theorem \ref{multitheo}.\\
In Section \ref{sec:6} we consider a simplified parametric model with constant parameters and equidistantly and synchronously observed data contaminated with market microstructure noise. The key result is the following Theorem \ref{boundtheo}: 
\begin{theo}\label{boundtheo}For a constant correlation coefficient $\rho$ and synchronously, equidistantly observed Brownian Motions with i.\,i.\,d.\,Gaussian noise local asymptotic normality (LAN) holds with $N^{-\nicefrac{1}{4}}$-rate. We conclude that $N^{\nicefrac{1}{4}}$ is a lower bound for the rate of convergence for any sequence of estimators and hence our multi-scale estimator is rate-optimal and asymptotically efficient. \end{theo}
In Proposition \ref{bound} we also give bounds for the asymptotic Fisher information.\\
Simulation results follow in Section \ref{sec:7}, where we compare the two proposed estimators and see that the multi-scale estimator also performs better in the case of finite sample sizes. We observe that if the influence of market microstructure effects is incisive the multi-scale approach provides a much more efficient estimation compared to the estimator based on subsampling. For the simulations we generated the observation times as arrivals of Poisson processes.

\section{Dealing with asynchronicity: Synchronizing data and covariance estimation for observations without noise\label{sec:3}}
\cite{hy} proved the consistency of their estimator
$$\widehat{\langle \tilde X,\tilde Y\rangle}_T^{(HY)}=\sum_{i=1}^{n}\sum_{j=1}^{m}\d \tilde X_{t_i}\d \tilde Y_{\tau_j}\1_{[\min{(t_{i},\tau_{j})}>\max{(t_{i-1},\tau_{j-1})}]}~,$$
where the product terms include all increments over overlapping time intervals, for the case where we observe the efficient price processes without noise.\\ In the case of high-frequency data contaminated by market microstructure noise, however, the estimator becomes inconsistent and explodes (tends to infinity) for $\min{\left(n,m\right)}\rightarrow \infty$.
We will prove this in Proposition \ref{propdiv}, but we focus first on an alternative useful method to deal with the asynchronicity of the data. This method was presented in \cite{palandri}\,(which he calls pseudo-aggregation).\\ For this purpose we construct a set with one or more than one observation times as elements which we call the joint grid. Those elements will be denoted by $\mathcal{H}^i$ and $\mathcal{G}^i, i \in \{0,\ldots,N\}$, and $(N+1)$ is the number of sets contained in the joint grid. The resulting synchronous realized covariance estimator for the integrated covariance will coincide with the one of \cite{hy} but this approach will be useful for our analysis of noise terms. In particular, this construction will enable us  to deal with the noise contamination by applying subsampling techniques in Section \ref{sec:4}. Based on this construction we get a joint grid with $(N+1)$ sets $\mathcal{H}^i$ and $\mathcal{G}^i$ where $N<\min{(n,m)}$. The last fact indicates heuristically that the efficiency of such techniques of covariance estimation depends on the number of observations available for the lower frequent process. In the ideal case both observation frequencies do not differ too much, i.\,e.\,$n$ and $m$ are of the same order and both assets show similar liquidity. In Assumption \ref{grid} we will give a more precise statement on the conditions imposed on the observation times needed for our analysis.\\
\\ The method of constructing a joint grid for the observations of both processes is described by the following iterative algorithm:
\begin{figure}[h!]\fbox{
\begin{minipage}[c]{\textwidth} 
first step:
\begin{itemize}\item for $t_{0}<\tau_{0}$ and $\tilde w_0\:=\min{(w \in\{1,\ldots,n\}|t_{w-1}<\tau_{0}\le t_w\})}$:
$$\mathcal{H}^0=\{t_{0},\ldots,t_{\tilde w_0}\}~~\mbox{and}~~\mathcal{G}^0=\{\tau_{0}\}$$
$$q_{1}\:= \begin{cases}\tilde w_0+1~~~\mbox{if}~~\tau_{0}=t_{\tilde w_0}\\ \tilde w_0 ~~~~~~~~~~\mbox{if}~~\tau_{0}<t_{\tilde w_0}\end{cases}~~\mbox{and}~~ r_{1}\:=1$$
\item for $t_{0}=\tau_{0}$:
$$\mathcal{H}^0=\{t_{0}\}~~\mbox{and}~~\mathcal{G}^0=\{\tau_{0}\}$$
$$q_{1}\:=1~~~\mbox{and}~~~r_1\:=1$$
\item for $t_{0}>\tau_{0}$ and $\tilde l_0\:=\min{(l \in \{1,\ldots,m\}|\tau_{l-1}<t_{0}\le \tau_l\})}$:
$$\mathcal{H}^0=\{t_{0}\}~~\mbox{and}~~\mathcal{G}^0=\{\tau_{0},\ldots,\tau_{\tilde l_0}\}$$
$$q_{1}\:=1~~\mbox{and}~~r_1\:=\begin{cases}\tilde l_0+1~~~\mbox{if}~~t_{0}=\tau_{\tilde l_0}\\ \tilde l_0 ~~~~~~~~~~\mbox{if}~~t_{0}<\tau_{\tilde l_0}\end{cases}$$
\end{itemize}
i-th step (given $\mathcal{H}^{i-1}$ and $\mathcal{G}^{i-1}$):
\begin{itemize}\item for $t_{q_i}<\tau_{r_i}$ and $\tilde w_i\:=\min{(w \in\{q_i+1,\ldots,n\}|t_{w-1}<\tau_{r_i}\le t_w\})}$:
$$\mathcal{H}^i=\{t_{q_i},\ldots,t_{\tilde w_i}\}~~\mbox{and}~~\mathcal{G}^i=\{\tau_{r_i}\}$$
$$q_{i}\dashrightarrow\begin{cases}q_{i+1}=\tilde w_i+1~~~\mbox{if}~~\tau_{r_i}=t_{\tilde w_i}\\ q_{i+1}=\tilde w_i ~~~~~~~~~~\mbox{if}~~\tau_{r_i}<t_{\tilde w_i}\end{cases}~~\mbox{and}~~ r_{i}\dashrightarrow r_{i+1}=r_i+1$$
\item for $t_{q_i}=\tau_{r_i}$:
$$\mathcal{H}^i=\{t_{q_i}\}~~\mbox{and}~~\mathcal{G}^i=\{\tau_{r_i}\}$$
$$q_{i}\dashrightarrow q_{i+1}=q_i+1~~~\mbox{and}~~~r_i\dashrightarrow r_{i+1}=r_i+1$$
\item for $t_{q_i}>\tau_{r_i}$ and $\tilde l_i\:=\min{(l \in \{r_i+1,\ldots,m\}|\tau_{l-1}<t_{q_i}\le \tau_l\})}$:
$$\mathcal{H}^i=\{t_{q_i}\}~~\mbox{and}~~\mathcal{G}^i=\{\tau_{r_i},\ldots,\tau_{\tilde l}\}$$
$$q_{i}\dashrightarrow q_{i+1}=q_i+1~~\mbox{and}~~r_i\dashrightarrow\begin{cases}r_{i+1}=\tilde l_i+1~~~\mbox{if}~~t_{q_i}=\tau_{\tilde l_i}\\ r_{i+1}=\tilde l_i ~~~~~~~~~~\mbox{if}~~t_{q_i}<\tau_{\tilde l_i}\end{cases}$$
\end{itemize}
\end{minipage}}
\end{figure}
\enlargethispage*{1cm}

\newpage
Let us give an example for the construction of the joint grid. Of course the example is just for illustration and the number of observations $m=6$ and $N=n=5$ is restricted and much smaller than in practice. The example emphasizes some important issues appearing when observations are asynchronous. \\ \\ \\
\begin{figure}[h!]
\textbf{Example}\begin{center}
\includegraphics[width=12cm]{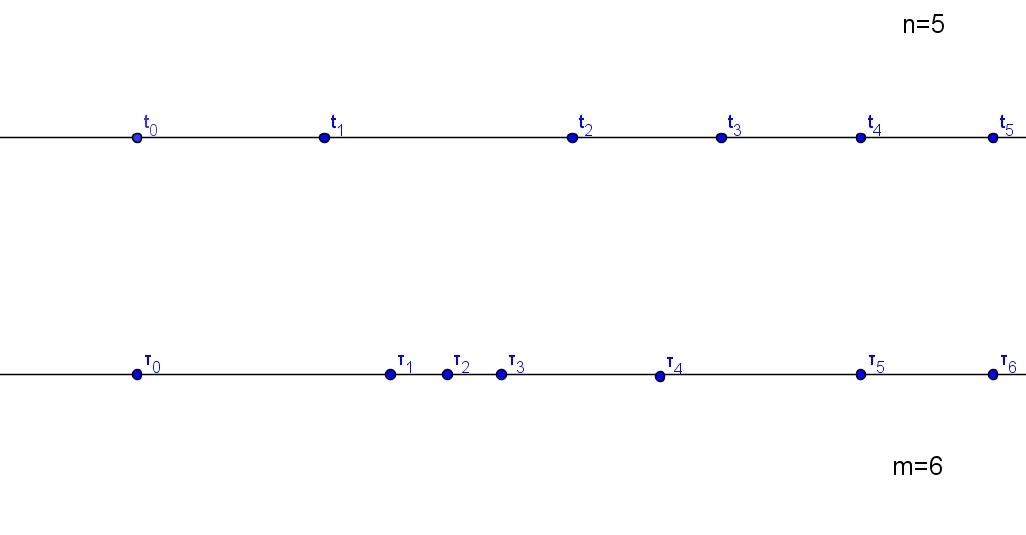}\end{center}
\textit{In the example illustrated above we have}
$\mathcal{H}^0=\{t_0\},\mathcal{G}^0=\{\tau_0\}, 
\mathcal{H}^1=\{t_1,t_2\},\mathcal{G}^1=\{\tau_1\},
\mathcal{H}^2=\{t_2\},\mathcal{G}^2=\{\tau_2,\tau_3,\tau_4\},
\mathcal{H}^3=\{t_3\},\mathcal{G}^3=\{\tau_4,\tau_5\},
\mathcal{H}^4=\{t_4\},\mathcal{G}^4=\{\tau_5\},
\mathcal{H}^5=\{t_5\},\mathcal{G}^5=\{\tau_6\}$~.
\textit{The example shows the important fact that the sets $\mathcal{H}^i$ and $\mathcal{G}^i$ are in general not disjoint and the maxima of consecutive sets can be the same time points. The minimum of a successive set can as well equal the maximum of the prevenient. For further examples see \cite{palandri}.}
\end{figure}\\
Next we pass over from the original observations to the sums of observed \textbf{increments} of the noisy log-prices over sets $\mathcal{H}^i$ and $\mathcal{G}^i$, respectively:
$$X^{\mathcal{H}^{i}}\:=\sum_{{t_j}\in \mathcal{H}^i}\d X_{t_j}~,~~Y^{\mathcal{G}^{i}}\:=\sum_{{\tau_j}\in \mathcal{G}^i}\d Y_{\tau_j}~,~~i \in \{0,\ldots,N\}~.$$
We observe that the realized covariance of the synchronized observations
$$\mathbf{SRC}\:=\sum_{i=0}^{N} X^{\mathcal{H}^{i}} Y^{\mathcal{G}^{i}}=\sum_{i=1}^{n}\sum_{j=1}^{m}\d \tilde X_{t_i}\d \tilde Y_{\tau_j}\1_{[\min{(t_{i},\tau_{j})}>\max{(t_{i-1},\tau_{j-1})}]}$$
is the well-known Hayashi-Yoshida estimator. We use a different representation of this estimator compared to \cite{palandri} using telescoping sums. If we define 
\begin{align*}&\mu_i\:=\max{(k|t_k \in \mathcal{H}^i)},&\tilde\mu_i\:=\max{(k|\tau_k \in \mathcal{G}^i)}&~~\mbox{and}\\ &\nu_i\:=\min{(k|t_k \in \mathcal{H}^i)},&\tilde\nu_i\:=\min{(k|\tau_k \in \mathcal{G}^i)}\,&~~,i\in\{0,\ldots,N\}\end{align*} and for the purpose of a simpler notation
\begin{align*}&X_{g_i}\:=X_{t_{\mu_i}},&Y_{ \gamma_i}\:=Y_{\tau_{\tilde\mu_i}}~~~~\,&,i\in\{0,\ldots,N\}~\mbox{and}\\
&X_{l_i}\,\:=X_{t_{\nu_i-1}},&Y_{\lambda_i}\:=Y_{\tau_{\tilde\nu_i-1}}~~&,i\in\{1,\ldots,N\}\end{align*}
with $l_0\:=t_0,\,\lambda_0\:=\tau_0$ we can write $X^{\mathcal{H}^{i}}$ and $Y^{\mathcal{G}{i}}$ as telescoping sums 
$X^{\mathcal{H}^{i}}=\left(X_{g_i}-X_{l_i}\right)$,
 $Y^{\mathcal{G}{i}}=\left(Y_{\gamma_i}-Y_{\lambda_i}\right)$
which leads to
$$\widehat{\langle \tilde X,\tilde Y\rangle}_T^{(HY)}=\sum_{i=0}^N\left(X_{g_i}-X_{l_i}\right)\left(Y_{\gamma_i}-Y_{\lambda_i}\right)~.$$
In this notation $g_i$ denotes the greatest and $l_i$ the last observation time before the least element of the set $\mathcal{H}^i$ and analogously $\gamma_i$ and $\lambda_i$ of $\mathcal{G}^i$.\\
We are interested in asymptotics when $n,\,m\rightarrow\infty$ being of the same order and the time lags between returns tending to zero and hence impose the following assumption on the observation design:
\begin{annahme}\label{grid}For the time lags between the observations 
$$\delta_N^X\:=\sup_{i\in\{1,\ldots,n\}}{\left(t_i-t_{i-1}\right)}=\mathcal{O}\left(\frac{1}{N}\right)~,$$
$$\delta_N^Y\:=\sup_{j\in\{1,\ldots,m\}}{\left(\tau_j-\tau_{j-1}\right)}=\mathcal{O}\left(\frac{1}{N}\right)$$
holds.\end{annahme}
By this assumption we exclude data where the lengths of the observed time intervals vary vigorously or the number of observations $n$ and $m$ are of different order. The assumption seems not too restrictive for most applications. Similar conditions are often imposed for asymptotic analysis, see e.\,g.\,\cite{zhangeps}. Further results are deduced under asymptotics for $n,m\rightarrow\infty$ and Assumption \ref{grid}. In the following we will use the notation $\E_{\tilde X,\tilde Y}[\,\cdot\,]\:=\E\left[\,\cdot\,|\tilde X,\tilde Y\right]$ for the conditional expectation given the paths of both efficient processes.
\begin{prop}\label{propdiv}Under Assumptions \ref{eff},\ref{e} and \ref{grid} for the synchronized realized covariance estimator 
$$\E_{\tilde X,\tilde Y}\left[\widehat{\langle \tilde X,\tilde Y\rangle}_T^{(HY)}\right]=\langle \tilde X,\tilde Y\rangle_T~~,~~\var_{\tilde X,\tilde Y}\left(\widehat{\langle \tilde X,\tilde Y\rangle}_T^{(HY)}\right)=\mathcal{O}_p\left(N\right)$$
holds.\end{prop}
\begin{proof} Unbiasedness follows directly from Assumption \ref{e} and the conditional variance can be simplified to:
\begin{align*}&\var_{\tilde X,\tilde Y}\left(\widehat{\langle \tilde X,\tilde Y\rangle}_T^{(HY)}\right)=\E\left[\sum_{i=0}^N\left(\epsilon_{g_i}^X-\epsilon_{l_i}^X\right)\left(\epsilon_{\gamma_i}^Y-\epsilon_{\lambda_i}^Y\right)\right]^2+\mathcal{O}_p(1)=\\ &~\E\left[\sum_{i=0}^N\left(\epsilon_{g_i}^X-\epsilon_{l_i}^X\right)^2\left(\epsilon_{\gamma_i}^Y-\epsilon_{\lambda_i}^Y\right)^2\right]+\mathcal{O}_p(1)=4N\eta_X^2\eta_Y^2+\mathcal{O}_p(1)=\mathcal{O}_p(N)\,.\end{align*}
The variances of $\sum_i\left(\epsilon_{g_i}^X-\epsilon_{l_i}^X\right)\left(\tilde Y_{\gamma_i}-\tilde Y_{\lambda_i}\right)$ and the second sum including increments of $\tilde X$ and $\epsilon^Y$ lead to the term of order 1 in propability. The mixed terms in the remaining second moment have an expectation equal to zero although consecutive sets $\mathcal{H}^{i}$ and $\mathcal{H}^{i+1}$ (or $\mathcal{G}^{i}$ and $\mathcal{G}^{i+1}$) are not generally disjoint. Nevertheless, our synchronization method was defined such that if the intersection of $\mathcal{H}^{i}$ and $\mathcal{H}^{i+1}$ is non-empty, $\mathcal{G}^{i} \cap \mathcal{G}^{i+1}=\emptyset$ holds. Assumption \ref{e} yields that each summand has expectation $2\eta_X^2\cdot 2\eta_Y^2$.
\end{proof}
We conclude that in the presence of market microstructure frictions the Hayashi-Yoshida estimator cannot yield a consistent estimation of the integrated covariance of the underlying efficient processes.
\newpage

\section{Dealing with noise contamination: the subsample estimator\label{sec:4}}
In this section we show that a subsample approach as presented in \cite{zhangmykland} leads to a consistent integrated covariance estimator with $N^{\nicefrac{1}{6}}$-rate of convergence. The  sets $\mathcal{H}^i$ and $\mathcal{G}^i$ are grouped in $K_N$ subsamples and for each subsample the (lower-frequency) realized covariances are calculated. So the synchronized observations are arranged to $K_N$ subsets of observations and for each subset we can calculate a realized covariance for which the error due to noise is smaller than for the highest-frequency realized covariance because of Proposition \ref{propdiv}. Averaging the realized covariations calculated with lower frequencies leads to the resulting estimator:
\begin{align}\label{subsamplingestimator}\widehat{\langle \tilde X,\tilde Y\rangle}_T^{sub}=\frac{1}{K_N}\sum_{i=K_N}^N\left(X_{g_i}-X_{l_{i-K_N}}\right)\left(Y_{\gamma_i}-Y_{\lambda_{i-K_N}}\right)~.\end{align}
Recall that we use the synchronized data and the joint grid that we presented in Section \ref{sec:3} to calculate the subsampling estimator. $K$ depends on $N$, but we drop the index in the following. The sum over $K$ subsamples and the sum over all increments in each subsample can be simplified to the form of the estimator stated above if we assume a regular allocation to subsamples ($\mathcal{H}^0,\mathcal{H}^{K},\ldots$ in the first subsample, $\mathcal{H}^1,\mathcal{H}^{K+1},\ldots$ in the second, etc.\,). 
We do not consider boundary effects (we focus on asymptotics $K=\KLEINO(N)$ and $K,n\rightarrow \infty$) and hence leave out the weights imposed by \cite{palandri}.
His (realized) covariance estimator is defined as an average of weighted sums of products over overlapping increments also divided in subsamples. If we regard the unions
$$A_{v,w}=\bigcup_{i=(w-1)K+v}^{wK+(v-1)}\mathcal{H}^i~~~\mbox{and}~~~B_{v,w}=\bigcup_{i=(w-1)K+v}^{wK+(v-1)}\mathcal{G}^i$$ of $K$ sets $\mathcal{H}^i$ and $\mathcal{G}^i$, respectively,
and $X_{v,w}=\sum_{t_j\in A_{v,w}}\d X_{t_j}$ and $Y_{v,w}=\sum_{\tau_j\in B_{v,w}}\d Y_{\tau_j}$, respectively,
his `Consistent Realized Covariance' estimator
$$\mathbf{CRC}=\frac{1}{K}\sum_{v=0}^{K-1} w_v\sum_{w=1}^{N/K-1}X_{v,w}Y_{v,w}$$
corresponds to our proposed estimator (except the weights $w_v$). Assuming that the noise processes across assets are independent, this estimator is unbiased and as we will prove in the following has for optimal choice $K=\mathcal{O}(N^{\nicefrac{2}{3}})$ an asymptotic variance of order $N^{-\nicefrac{1}{3}}$. The unbiasedness is the reason why we do not need a bias-correction term in contrast to the realized variance case (two time-scales estimator presented by \cite{zhangmykland}).\\
\cite{palandri} has proved this result for his similar estimator, but it is reasonable for our further analysis concerning the multi-scale estimator in chapter \ref{sec:5} to give a short calculation of the asymptotic variance of the subsample estimator in our illustration. We are only interested in the order of the asymptotic variance and we impose mild assumptions on the grid and time intervals that are inherent in the method of subsampling and the data. In particular, we assume a regular allocation to subsamples as stated above and Assumption \ref{grid} that ensures (together with Assumption \ref{eff})  $\d \tilde X_{t_i}=\mathcal{O}_p(\sqrt{1/N})$, $\d \tilde Y_{\tau_j}=\mathcal{O}_p(\sqrt{1/N})$.\\ 
The total variance can be written as
\begin{multline}\label{split}\var\left(\widehat{\langle \tilde X,\tilde Y\rangle}_T^{sub}\right)=\frac{1}{K^2}\sum_{i=K}^N\sum_{j=K}^N\cov\Big(\left(X_{g_i}-X_{l_{i-K}}\right)\left(Y_{\gamma_i}-Y_{\lambda_{i-K}}\right)\,,\\ ~~~~~~~~~~~~~~~~~~~~~~~~~~~~~~~~~~~~~~~~~~~~~~~~~~~~~~~~~~~~~~~\,\left(X_{g_j}-X_{l_{j-K}}\right)\left(Y_{\gamma_j}-Y_{\lambda_{j-K}}\right)\Big)\\
=\frac{1}{K^2}\sum_{i=K}^N\sum_{j=K}^N\left(\cov\left( \ec_i ,\ec_j\right)+\cov\left( \mc_i ,\mc_j\right)+\cov\left( \mic_i ,\mic_j\right)+\cov\left( \nc_i ,\nc_j\right)\right)\end{multline}
with the four uncorrelated terms:
\begin{align*}\ec_i&=\left(\tilde X_{g_i}-\tilde X_{l_{i-K}}\right)\left(\tilde Y_{\gamma_i}-\tilde Y_{\lambda_{i-K}}\right)~,\\
\mc_i&=\left(\tilde X_{g_i}-\tilde X_{l_{i-K}}\right)\left(\epsilon^Y_{\gamma_i}-\epsilon^Y_{\lambda_{i-K}}\right)~,\\
\mic_i&=\left(\epsilon^X_{g_i}-\epsilon^X_{l_{i-K}}\right)\left(\tilde Y_{\gamma_i}-\tilde Y_{\lambda_{i-K}}\right)~,\\
\nc_i&=\left(\epsilon^X_{g_i}-\epsilon^X_{l_{i-K}}\right)\left(\epsilon^Y_{\gamma_i}-\epsilon^Y_{\lambda_{i-K}}\right)~.
\end{align*}
We will consider the four summands consecutively. \\
We start our analysis of the asymptotic orders of the different summands in the total variance focusing on the sum of covariances $\cov\left(\ec_i\,,\,\ec_j\right)$. This is the variance due to discretization and would be the total variance of a subsampling estimator calculated with observations of the efficient processes without noise. First we deduce the order of increments for the efficient processes from Assumptions \ref{eff} and \ref{grid}.
\begin{lem} If Assumptions \ref{eff} and \ref{grid} hold, we obtain the following asymptotic orders for the efficient processes without microstructure noise:
\begin{subequations}
\begin{align}\label{po1}\E\left[\left(\txg-\txl\right)^2\right]&=\mathcal{O} \left(K/N\right)~, \\
							\label{po2}\E\left[\left(\tyg-\tyl\right)^2\right]&=\mathcal{O} \left(K/N\right)~, \\
							\label{po3}\E\left[\left(\txg-\txl\right)\left(\tyg-\tyl\right)\right] &=\mathcal{O} \left(K/N\right)~, \\ 
							\label{po4}\E\left[\left(\txg-\txl\right)^2\left(\tyg-\tyl\right)^2\right] &=\mathcal{O} \left(K^2/N^2\right)~. \end{align}
							\end{subequations}\end{lem}
\enlargethispage*{1cm}
\begin{proof}
By Assumption \ref{eff} the drifts of the efficient processes are bounded and thus
$$\E\left[\left(\txg-\txl\right)^2\right]=\E\left[\left(\int_{l_{i-K}}^{g_i} \sigma_t^X~dB_t^X\right)^2\right]+\KLEINO\left(\frac{K}{N}\right)$$
holds because the squared drift term is of order $K^2/N^2$ and the mixed term is of order $(K/N)^{\nicefrac{3}{2}}$. Therefore, to prove \eqref{po1} it suffices to apply It\^{o} isometry and the mean value theorem using again Assumption \ref{eff} for the spot volatilities:\begin{align*}\E\left[\left(\int_{l_{i-K}}^{g_i} \sigma_t^X~dB_t^X\right)^2\right]&= \E\left[\int_{l_{i-K}}^{g_i}\left(\sigma_t^X\right)^2~dt\right]\\ &=\sigma_{l_{i-K},g_i}^{MVT} \left(g_i-l_{i-K}\right)=\mathcal{O}\left(\frac{K}{n}\right)~.\end{align*}
The asymptotic orders of the time increments are given by Assumption \ref{grid}. The constant $\sigma_{l_{i-K},g_i}^{MVT}$ occuring by application of the mean value theorem is finite because of Assumption \ref{eff}. The proof of \eqref{po2} follows analogously.\\
Using \eqref{po1} and \eqref{po2} we obtain \eqref{po3} by the Cauchy-Schwarz inequality:
\begin{align*}\E\left[\left(\txg-\txl\right)\left(\tyg-\tyl\right)\right]&\le \sqrt{\E\left[\left(\txg-\txl\right)^2\right]}\cdot \sqrt{\E\left[\left(\tyg-\tyl\right)^2\right]}\\ &=\mathcal{O}\left(\frac{K}{n}\right)~.\end{align*}
The fourth moments of the increments can be bounded by the squared quadratic covariation using the Burkholder-Davis-Gundy inequality, so it is adequate to prove \eqref{po4} again using the Cauchy-Schwarz inequality:
\begin{align*}\E\left[\left(\txg-\txl\right)^2\left(\tyg-\tyl\right)^2\right]\le \sqrt{\E\left[\left(\txg-\txl\right)^4\right]}\cdot \sqrt{\E\left[\left(\tyg-\tyl\right)^4\right]}\\
\le \sqrt{c\,\E\left(\langle\tilde X\,,\,\tilde X\rangle_{g_i}-\langle\tilde X\,,\,\tilde X\rangle_{l_{i-k}}\right)^2}\cdot \sqrt{c^*\,\E\left(\langle\tilde Y\,,\,\tilde Y\rangle_{\gamma_i}-\langle\tilde Y\,,\,\tilde Y\rangle_{\lambda_{i-K}}\right)^2}=\mathcal{O}\left(\frac{K^2}{N^2}\right)
\end{align*}
with constants $c$ and $c^{*}$ from the application of the Burkholder-Davis-Gundy inequality. The order $K^2/N^2$ then easily follows e.\,g.\,using again the mean value theorem as above.
\end{proof}

\begin{cor}We obtain for the variance due to discretization, when there is no microstructure noise present, that we denote by $\var_{\eta_X=\eta_Y=0}\left(\,\cdot\,\right)$:
\begin{align*}\var_{\eta_X=\eta_Y=0}\left(\widehat{\langle \tilde X,\tilde Y\rangle}_T^{sub}\right)=\mathcal{O}\left(\frac{K}{N}\right)~.\end{align*}\end{cor}
\begin{proof}
\begin{align*}\var_{\eta_X=\eta_Y=0}\left(\widehat{\langle \tilde X,\tilde Y\rangle}_T^{sub}\right)=\frac{1}{K^2}\sum_{i=K}^N\sum_{j=K}^N\cov\left(\ec_i\,,\,\ec_j\right)=~~~~~~~~~~~~~~~~~~~~~~~~~~~~~~~~~~~~~~~~~~~~~~~~~~~~~~~~~\\
=\frac{2}{K^2}\sum_{i=K}^N\sum_{j=K}^{i-1}\cov\left(\left(\tilde X_{g_i}\hspace*{-0.1cm}- \hspace*{-0.1cm}\tilde X_{l_{i-K}}\right)\left(\tilde Y_{\gamma_i}\hspace*{-0.1cm}-\hspace*{-0.1cm}\tilde Y_{\lambda_{i-K}}\right)\,,\,\left(\tilde X_{g_j}\hspace*{-0.1cm}- \hspace*{-0.1cm}\tilde X_{l_{j-K}}\right)\left(\tilde Y_{\gamma_j}\hspace*{-0.1cm}-\hspace*{-0.1cm}\tilde Y_{\lambda_{j-K}}\right)\right)~~~~~\\
+\frac{1}{K^2}\sum_{i=K}^N\var\left(\ec_i\right)\\
=\frac{2}{K^2}\sum_{i=K}^N\sum_{j=K}^{i-1}\1_{\{|i-j|\le K\}}\cov\Big(\left(\tilde X_{g_i}\hspace*{-0.1cm}-\hspace*{-0.1cm}\tilde X_{g_j}+\tilde X_{g_j}\hspace*{-0.1cm}-\hspace*{-0.1cm}\tilde X_{l_{i-K}}\right)\left(\tilde Y_{\gamma_i}\hspace*{-0.1cm}-\hspace*{-0.1cm}\tilde Y_{\gamma_j}\hspace*{-0.1cm}+\hspace*{-0.1cm}\tilde Y_{\gamma_j}\hspace*{-0.1cm}-\hspace*{-0.1cm}\tilde Y_{\lambda_{i-K}}\right)\,,~~\\
\,\left(\tilde X_{g_j}\hspace*{-0.1cm}-\hspace*{-0.1cm}\tilde X_{l_{i-K}}\hspace*{-0.1cm}+\hspace*{-0.1cm}\tilde X_{l_{i-K}}\hspace*{-0.1cm}-\hspace*{-0.1cm}\tilde X_{l_{j-K}}\right)\left(\tilde Y_{\gamma_j}\hspace*{-0.1cm}-\hspace*{-0.1cm}\tilde Y_{\lambda_{i-K}}\hspace*{-0.1cm}+\hspace*{-0.1cm}\tilde Y_{\lambda_{i-K}}\hspace*{-0.1cm}-\hspace*{-0.1cm}\tilde Y_{\lambda_{j-K}}\right)\Big)+\frac{1}{K^2}\sum_{i=K}^N\var\left(\ec_i\right)\\
=\frac{2}{K^2}\sum_{i=K}^N\sum_{j=K}^{i-1}\1_{\{|i-j|\le K\}}\var\left(\left(\tilde X_{g_j}\hspace*{-0.1cm}-\hspace*{-0.1cm}\tilde X_{l_{i-K}}\right)\left(\tilde Y_{\gamma_j}\hspace*{-0.1cm}-\hspace*{-0.1cm}\tilde Y_{\lambda_{i-K}}\right)\right)+\frac{1}{K^2}\sum_{i=K}^N\var\left(\ec_i\right)~~~~~~~\\
=\mathcal{O}\left( \frac{1}{K^2}NK\frac{K^2}{N^2}\right)=\mathcal{O}\left(\frac{K}{N}\right)~~.~~~~~~~~~~~~~~~~~~~~~~~~~~~~~~~~~~~~~~~~~~~~~~~~~~~~~~~~~~~~~~~~~~~~~~~~~~~~~~~~~~~~~~
\end{align*}
In this calculation we also used characteristics of our synchronization method. The increments $\left(\tilde X_{g_i}- \tilde X_{l_{i-K}}\right)$ and $\left(\tilde X_{g_j}- \tilde X_{l_{j-K}}\right)$ are non-overlapping and hence uncorrelated for $|i-j|>K$. Taking the construction procedure of the joint grid into account, increments $\left(\tilde X_{g_i}- \tilde X_{l_{i-K}}\right)$ and $\left(\tilde Y_{\gamma_j}-\tilde Y_{\lambda_{j-K}}\right)$ are uncorrelated as well for $|i-j|>K$.\end{proof}
So, the asymptotic order of the discretization variance is
\begin{align}\label{disvar}\frac{1}{K^2}\sum_{i=K}^N\sum_{j=K}^N\cov\left(\ec_i\,,\,\ec_j\right)=\mathcal{O}\left(\frac{K}{N}\right).\end{align}
For the mixed summands in \eqref{split} we obtain under Assumption \ref{eff}, \ref{e} and \ref{grid}
\begin{align}\notag\frac{1}{K^2}\sum_{i=K}^N\sum_{j=K}^N\cov\left(\mc_i\,,\,\mc_j\right)=\frac{1}{K^2}\sum_{i=K}^N\var\left(\mc_i\right)+\frac{2}{K^2}\sum_{i=K}^{N-1}\cov\left(\mc_i\,,\,\mc_{i+1}\right)\\ \le \frac{3}{K^2}\sum_{i=1}^N\var\left(\mc_i\right)=C_{\mathbf{m}}\frac{\eta_Y^2}{K}=\mathcal{O}\left(\frac{\eta_Y^2}{K}\right)\label{mixvar}~\end{align}\enlargethispage*{1cm}
with a constant $C_{\mathbf{m}}$ and analogously with a constant $C_{\nu}$
\begin{align}\label{mixvar1}\frac{1}{K^2}\sum_{i=K}^N\sum_{j=K}^N\cov\left(\mic_i\,,\,\mic_j\right)\le C_{\nu}\frac{\eta_X^2}{K}=\mathcal{O}\left(\frac{\eta_X^2}{K}\right)~.\end{align}
The fourth (noise) term in \eqref{split} is of order $N/K^2$ because
\begin{align}\notag\frac{1}{K^2}\sum_{i=K}^N\sum_{j=K}^N\cov\left(\nc_i\,,\,\nc_j\right)=\frac{1}{K^2}\sum_{i=K}^N\var\left(\nc_i\right)+\frac{2}{K^2}\sum_{i=K}^{N-1}\cov\left(\nc_i\,,\,\nc_{i+1}\right)\\ \le \frac{3}{K^2}\sum_{i=1}^N\var\left(\nc_i\right)=C_{\mathbf{n}}\frac{\eta_X^2\eta_Y^2 N}{K^2}=\mathcal{O}\left(\frac{\eta_X^2\eta_Y^2N}{K^2}\right)\label{noisevar}\end{align}
with a constant $C_{\mathbf{n}}$ holds.
We used Assumption \ref{e} that the noise is i.\,i.\,d.\,, but recall that $g_i=g_{i+1}$ and $l_i=g_{i-1}$ is possible.
The reason why we do not include moments of the distribution of the noise processes in the constants is that these distributions may depend on the number of observations $N$ although we disclaimed on further indices.\\
For a choice $K=\mathcal{O}(N^{\nicefrac{2}{3}})$ the first and fourth term are of the same order $N^{-\nicefrac{1}{3}}$ and thus we see that the $\widehat{\langle \tilde X,\tilde Y\rangle}_T^{sub}$-estimator is consistent with rate $N^{\nicefrac{1}{6}}$.\\
We summarize the properties of the subsample estimator in the following proposition.
\begin{prop}
If we choose $K_N=\mathcal{O}\left(N^{\nicefrac{2}{3}}\right)$ the subsample estimator 
\begin{align*}\widehat{\langle \tilde X,\tilde Y\rangle}_T^{sub}=\frac{1}{K_N}\sum_{i=K_N}^N\left(X_{g_i}-X_{l_{i-K_N}}\right)\left(Y_{\gamma_i}-Y_{\lambda_{i-K_N}}\right)\end{align*}
is a consistent unbiased estimator with asymptotic variance of order $N^{-\nicefrac{1}{3}}$:
\begin{subequations}
\begin{align}\E\left[\widehat{\langle \tilde X,\tilde Y\rangle}_T^{sub}-\langle \tilde X,\tilde Y\rangle_T\right]=0~,\end{align}
\begin{align}\var\left(\widehat{\langle \tilde X,\tilde Y\rangle}_T^{sub}\right)=\mathcal{O}\left(N^{-\nicefrac{1}{3}}\right)~.\end{align}
Bias-variance decomposition yields that
\begin{align}\E\left(\left[\left(\widehat{\langle \tilde X,\tilde Y\rangle}_T^{sub}-\langle \tilde X,\tilde Y\rangle_T\right)^2\right]\right)^{\nicefrac{1}{2}}=\mathcal{O}\left(N^{-\nicefrac{1}{6}}\right)~.\end{align}\end{subequations}
\end{prop}
In the next section we will show that a multi-scale approach can improve this $N^{\nicefrac{1}{6}}$ rate of convergence to $N^{\nicefrac{1}{4}}$.

\section{Upgrading the subsample estimator using a multi-scale approach\label{sec:5}}
In this section we show that using different lower frequencies for subsampling instead of one singular fixed $K_N$ and calculating a weighted mean of the different subsample estimators leads to a more efficient consistent estimator with a better rate of convergence $N^{\nicefrac{1}{4}}$. We calculate $M_N$ subsample estimators using the regular sequence  $i=1,2,3\ldots,M_N-1,M_N$ instead of the fixed $K_N$ in Section \ref{sec:4}. We remark that there is no advantage using a more general sequence of subsample frequencies. We focus on the variance of our multi-scale estimator for the integrated covariance due to the noise terms first, which is the conditional variance given the paths of both efficient processes, and calculate noise-optimal weights that minimize this variance due to market microstructure frictions. In the following we skip the index $N$ for $M$.
The general multi-scale estimator
$$\widehat {\langle \tilde X,\tilde Y\rangle}_T^{mult}=\sum_{i=1}^M\alpha_i\widehat{\langle \tilde X,\tilde Y\rangle}_T^{sub\,,\,i}=\sum_{i=1}^M\alpha_i\frac{1}{i}\sum_{j=i}^N\left(X_{g_j}-X_{l_{j-i}}\right)\left(Y_{\gamma_j}-Y_{\lambda_{j-i}}\right)$$
with weights $\alpha_i$ will give a consistent estimator by choosing the weights optimally.\\
To determine noise-optimal weights, we will impose side conditions on the weights that simplify the minimization problem for the variance due to noise. After that, we will prove that the variance due to mixed terms is asymptotically negligible and we will calculate the discretization variance. As for the subsample estimator there is a trade-off between variance due to noise and variance due to discretization. Choosing $M$ optimally in the way that the mean square error is minimized, we will see that the total variance is of order $N^{-\nicefrac{1}{2}}$. An important fact is that the weights as well as the order of $M$ and the rate of convergence of the multi-scale estimator are in line with the variance case presented by \cite{zhang}. We impose the condition
\begin{subequations}
\begin{equation}\label{unb}\sum_{i=1}^M\alpha_i=1~,\end{equation}
that ensures unbiasedness of the resulting estimator, and the auxiliary condition
\begin{equation}\sum_{i=1}^M\frac{\alpha_i}{i}=0~,\end{equation}
\end{subequations}
that will guarantee that the `leading' term in the variance equals zero, on the weights that gives
\begin{align*}&\sum_{i=1}^M\frac{\alpha_i}{i}\sum_{j=i}^N\left(\epsilon^X_{g_j}-\epsilon^X_{l_{j-i}}\right)\left(\epsilon^Y_{\gamma_j}-\epsilon^Y_{\lambda_{j-i}}\right)\\
&~~~=\sum_{i=1}^M\frac{\alpha_i}{i}\left(\sum_{j=0}^N\left(\epsilon^X_{g_j}\epsilon^Y_{\gamma_j}+\epsilon^X_{l_j}\epsilon^Y_{\lambda_j}\right)+R_N-\sum_{j=i}^N\left(\epsilon^X_{g_j}\epsilon^Y_{\lambda_{j-i}}+\epsilon^X_{l_{j-i}}\epsilon^Y_{\gamma_j}\right)\right)\\
&~~\stackrel{\text{(9b)}}{=}\sum_{i=1}^M\frac{\alpha_i}{i}\left(R_N-\sum_{j=i}^N\left(\epsilon^X_{g_j}\epsilon^Y_{\lambda_{j-i}}+\epsilon^X_{l_{j-i}}\epsilon^Y_{\gamma_j}\right)\right)
\end{align*}
with the remainder term $$R_N=-\sum_{j=0}^{i-1}\epsilon^X_{g_j}\epsilon^Y_{\gamma_j}-\sum_{j=N-i+1}^N\epsilon^X_{l_{j}}\epsilon^Y_{\lambda_j}$$ that is asymptotically negligible in the sum because of Assumption \ref{e} and $i\le M=\KLEINO(N)$. Hence we only have to focus on the residual term for the analysis of the variance due to noise.
Define
$$U_i\:=-\sum_{j=i}^N\left(\epsilon^X_{g_j}\epsilon^Y_{\lambda_{j-i}}+\epsilon^X_{l_{j-i}}\epsilon^Y_{\gamma_j}\right)~.$$
\begin{lem}\label{msc1}
The variance of $U_i$ satisfies for all $i \in \{1,\ldots,M\}$ the following asymptotic inequality:
\begin{equation}\label{absc}\var\left(U_i\right)\le 6 N\eta_X^2\eta_Y^2~.\end{equation}\end{lem}
\begin{proof}The summands have variances $2\eta_X^2\eta_Y^2$. Because of Assumption \ref{e} only covariances
$$\cov\left(\left(\epsilon^X_{g_j}\epsilon^Y_{\lambda_{j-i}}+\epsilon^X_{l_{j-i}}\epsilon^Y_{\gamma_j}\right)\,,\,\left(\epsilon^X_{g_k}\epsilon^Y_{\lambda_{k-i}}+\epsilon^X_{l_{k-i}}\epsilon^Y_{\gamma_k}\right)\right)$$
with $k=j\pm1$ can be non-zero. Those covariances are smaller or equal than $2\eta_X^2\eta_Y^2$ and we obtain the inequality by separating the variance of $U_i$ in the sum over all variances and covariances.\end{proof}
Under our assumptions the random variables $U_i$ are not necessarily uncorrelated. For different subsample frequencies $i$ and $k\in\{i-1,\ldots,i+1\}$ the covariances
$$\cov\left(\sum_{j=i}^N\epsilon^X_{g_j}\epsilon^Y_{\lambda_{j-i}}+\epsilon^X_{l_{j-i}}\epsilon^Y_{\gamma_j}\,,\,\sum_{r=k}^N\epsilon^X_{g_r}\epsilon^Y_{\lambda_{r-k}}+\epsilon^X_{l_{r-k}}\epsilon^Y_{\gamma_r}\right)$$
can be non-zero.
In fact, only addends with $r\in\{j-1,j,j+1\}$ and $(j-i)=(r-k)$ \textbf{could} have non-zero covariance (if the considered maximum \textbf{and} minimum are equal) and hence correlation effects will be very small but in any case a mathematical analysis gives an upper bound and the exact asymptotic order using an inequality similar to those in the last section.
\begin{lem} \label{msc2}For the variance of the general multi-scale estimator due to noise the asymptotic inequality
\begin{align}\var\left(\sum_{i=1}^M\frac{\alpha_i}{i}\sum_{j=i}^N\left(\epsilon^X_{g_j}-\epsilon^X_{l_{j-i}}\right)\left(\epsilon^Y_{\gamma_{j}}-\epsilon^Y_{\lambda_{j-i}}\right)\right)
\le\sum_{i=1}^M\frac{\alpha_i^2}{i^2}\,18\,N\,\eta_X^2\eta_Y^2 
\end{align} holds.\end{lem}
\begin{proof}Applying the Cauchy-Schwarz inequality to the covariance terms considered above and using inequality \eqref{absc} we conclude that 
\begin{align*}\var\left(\sum_{i=1}^M\frac{\alpha_i}{i}\sum_{j=i}^N\left(\epsilon^X_{g_j}-\epsilon^X_{l_{j-i}}\right)\left(\epsilon^Y_{\gamma_{j}}-\epsilon^Y_{\lambda_{j-i}}\right)\right)
&\le \sum_{i=1}^M\frac{\alpha_i^2}{i^2}\,3\,\var\left(U_i\right)\\ &\le\sum_{i=1}^M\frac{\alpha_i^2}{i^2}\,18\,\eta_X^2\eta_Y^2 \left(N+\KLEINO(N)\right)~
\end{align*}which gives the result of Lemma \ref{msc2}.\end{proof}
Minimization with side conditions yields for an arbitrary constant $c \in \R$:
\begin{align*} \frac{\partial}{\partial \alpha_i}\left(c\sum_j\frac{\alpha_j^2}{j^2}+\lambda_1\left(\sum_j \alpha_j-1\right)+\lambda_2\left(\sum_j\frac{\alpha_j}{j}\right)\right)&=0\\
\Leftrightarrow~~~~~~~~~~~~~~~~~~~~~~~~~~~~~~~~~~~~~~~~~~~~~~~~~~~~~~~~~~~~~\,  2c \frac{\alpha_i}{i^2}+\lambda_1+\frac{\lambda_2}{i}&=0\\
\Leftrightarrow~~~~~~~~~~~~~~~~~~~~~~~~~~~~~~~~~~~~~~~~~~~~~~~~~~~~~~~~~~~~~~~~~~~~~~~~~~~~~~~~~~~~ \alpha_i&=-\frac{1}{2c}\left(i^2\lambda_1+\lambda_2i\right)~.
\end{align*}
Since $1=\sum \alpha_i=-\frac{1}{2c}\left(\lambda_1\sum i^2+\lambda_2\sum i\right)$ and $0=\sum \frac{\alpha_i}{i}=-\frac{1}{2c}\left(\lambda_1\sum i+\lambda_2 M\right)$ we get the result
$$\lambda_1=\frac{-24c}{M^3-M}~~,~~\lambda_2=\frac{12c}{(M-1)M}~~,~~\alpha_i =\frac{12i^2}{M^3-M}-\frac{6i}{(M-1)M}~.$$
The noise-optimal weights are the same as for the MSRV-estimator invented by \cite{zhang} for efficient high-frequency realized variance estimation from noisy observations. This is a positive aspect of using these methods because the weights are calculated once and serve for variance as well as covariance estimation.\\
Inserting the noise-optimal weights
\begin{equation}\label{opt}\alpha_{i,opt}= \frac{12 i^2}{M^3}-\frac{6i}{M^2}\left(1+\KLEINO(1)\right)\end{equation}
in the noise-variance term above yields the result stated in the following proposition.
\begin{prop}
If we insert the noise-optimal weights $\alpha_{i,opt}$ from \eqref{opt} in the general multi-scale estimator and assume that the variances $\eta_X^2$, $\eta_Y^2$ of the noise distributions are of order 1, the asymptotic variance due to noise satisfies
\begin{align}\var\left(\sum_{i=1}^M\frac{\alpha_i}{i}\sum_{j=i}^N\left(\epsilon^X_{g_j}-\epsilon^Y_{\lambda_{j-i}}\right)\left(\epsilon^X_{l_{j-i}}-\epsilon^Y_{\gamma_j}\right)\right)=\mathcal{O}\left(\frac{N}{M^3}\right)~.\end{align}\end{prop}
\begin{proof} We calculate the minimal noise-variance using the optimal weights \eqref{opt}. For the occurring sums it suffices to use the asymptotic formula $\sum_{i=1}^l i^k=\frac{l^{k+1}}{k+1}+\KLEINO\left(l^{k+1}\right)$ and Lemma \ref{msc1} and \ref{msc2} to deduce the asymptotic order of the variance:
\begin{align*} \sum_{i=1}^M\frac{\left(\alpha_{i,opt}\right)^2}{i^2}\eta_X^2\eta_Y^2 N&= \sum_{i=1}^M\left(\frac{144i^2}{M^6}-\frac{144i}{M^5}+\frac{36}{M^4}\right)\eta_X^2\eta_Y^2 N\\
&=\frac{12N}{M^3}\eta_X^2\eta_Y^2+\KLEINO\left(\frac{N}{M^3}\right)~.\end{align*}
\end{proof}
We have shown that our resulting multi-scale estimator with noise-optimal weights has a variance due to noise contamination of asymptotic order $N/M^3$. Next we focus on the other terms occurring in the total variance.\\
Under the stated assumptions the multi-scale estimator for integrated covariance is unbiased and the variance induced by the mixed summands is asymptotically negligible. Unbiasedness holds by Condition \eqref{unb} on the weights and we focus on the variance of the mixed summands now.\\
The analysis of the terms with the $\mc_i$s and $\mic_i$s (see \eqref{split} for definition) is analogous and we only mention the analysis of the first term. Inserting the noise-optimal weights \eqref{opt} the variance equals
\begin{align*}
&\var\left(\sum_{i=1}^M\frac{\alpha_{i,opt}}{i}\sum_{j=i}^N\left(\tilde X_{g_j}-\tilde X_{l_{j-i}}\right)\left(\epsilon^Y_{\gamma_j}-\epsilon^Y_{\lambda_{j-i}}\right)\right)\\
&=2\sum_{i=1}^M\sum_{k=1}^i\left(\frac{144ik}{M^6}+\frac{36}{M^4}-\frac{72i}{M^5}-\frac{72k}{M^5}\right)\\ &~~~~~~~~~~~~~~~~~\times \sum_{j=i}^N\sum_{r=k}^N\cov\left(\left(\tilde X_{g_j}-\tilde X_{l_{j-i}}\right)\left(\epsilon^Y_{\gamma_j}-\epsilon^Y_{\lambda_{j-i}}\right)\,,\,\left(\tilde X_{g_r}-\tilde X_{l_{r-k}}\right)\left(\epsilon^Y_{\gamma_r}-\epsilon^Y_{\lambda_{r-k}}\right)\right)\,.\end{align*}
The covariances can only be non-zero if the time intervals of the increments of the efficient process $\tilde X$ are overlapping and because of Assumption \ref{e} if $r\in\{j-1,j,j+1\}$ or $(r-k)=(j-i)$ holds. Therefore, we conclude that a constant $C^{*}$ exists such that
\begin{align*}2\sum_{i=1}^M\sum_{k=1}^i\left(\frac{144ik}{M^6}+\frac{36}{M^4}-\frac{72i}{M^5}-\frac{72k}{M^5}\right) N \frac{i}{N} \,2\, C^{*}\eta_Y^2\end{align*}
is an upper bound for the variance and we obtain the asymptotic order $1/M$ for the mixed terms. We have deduced that the mixed terms are negligible in the asymptotic total variance and
hence we focus next on the terms containing the $\ec_i$s (see \eqref{split}) and the variance due to discretization.\\
Thus the variance term of interest is \enlargethispage*{2cm}
\begin{multline*}\var\left(\sum_{i=1}^M\frac{\alpha_i}{i}\sum_{j=i}^N\left(\tilde X_{g_j}-\tilde X_{l_{j-i}}\right)\left(\tilde Y_{\gamma_j}-\tilde Y_{\lambda_{j-i}}\right)\right)=\\
\phantom{\var~~~}2\sum_{i=1}^M\sum_{k=1}^i\frac{\alpha_i\alpha_k}{ik}\cov\big(\sum_{j=i}^N\left(\tilde X_{g_j}-\tilde X_{l_{j-i}}\right)\left(\tilde Y_{\gamma_j}-\tilde Y_{\lambda_{j-i}}\right)\,,~~~~~~~~~~~~~~~~~~~~~~~~~~~~~~~~\\ ~~~~~~~~~~~~~~~~~~\,\sum_{r=k}^N\left(\tilde X_{g_r}-\tilde X_{l_{r-k}}\right)\left(\tilde Y_{\gamma_r}-\tilde Y_{\lambda_{r-k}}\right)\big)~.
\end{multline*}
Considering next the single covariance terms in the discretization-variance using
$$\E\left[\left(\tilde X_{t_i}-\tilde X_{t_{i-l}}\right)\left(\tilde X_{t_j}-\tilde X_{t_{j-k}}\right)\right]=\E\left(\tilde X_{t_{\min{(i,j)}}}-\tilde X_{t_{\max{(i-l,j-k)}}}\right)^2\1_{\{\min{(i,j)}>\max{(i-l,j-k)}\}}$$
for arbitrary $i,j$ and $l,k$ leads to
\begin{align*}\var\left(\sum_{i=1}^M\frac{\alpha_i}{i}\sum_{j=i}^N\left(\tilde X_{g_j}-\tilde X_{l_{j-i}}\right)\left(\tilde Y_{\gamma_j}-\tilde Y_{\lambda_{j-i}}\right)\right)~~~~~~~~~~~~~~~~~~~~~~~~~~\\ \le 2\sum_{k=1}^M\sum_{l=1}^k\left(\frac{144lk}{M^6}+\frac{36}{M^4}-\frac{72l}{M^5}-\frac{72k}{M^5}\right)C^{*}\cdot\frac{l^2k}{N}\le C^{**}\frac{M}{N}\end{align*}
with constants $C^{*}$ and $C^{**}$. The inequality is deduced in the usual way analyzing which increments are overlapping and hence correlated and using the asymptotic orders of the increments known by Assumptions \ref{eff} and \ref{grid}. 
\begin{prop} For the variance of the noise-optimal multi-scale estimator due to discretization the following asymptotic inequality holds:
\begin{equation}\var\left(\sum_{i=1}^M\frac{\alpha_i}{i}\sum_{j=i}^N\left(\tilde X_{g_j}-\tilde X_{l_{j-i}}\right)\left(\tilde Y_{\gamma_j}-\tilde Y_{\lambda_{j-i}}\right)\right)=\mathcal{O}\left( \frac{M}{N}\right)~.\end{equation}\end{prop}
The discretization variance is of order $M/N$ and we have to choose
$$M=\mathcal{O}(\sqrt{N})$$to reduce the total variance to order $1/M$ or rather $1/\sqrt{N}$. There is a trade-off between the variance terms due to microstructure noise and discretization and the total variance is minimized by a choice of $M$ that induces both being of the same asymptotic order. Calculating $M=\mathcal{O}(\sqrt{N})$ different subsample estimators and calculating the weighted sum with noise-optimal weights results in obtaining an estimator with asymptotic total variance of order $1/\sqrt{N}$ upgrading the rate of convergence to $N^{\nicefrac{1}{4}}$ compared with $N^{\nicefrac{1}{6}}$ for the simple one scale estimator presented in Section \ref{sec:3}. Although the new estimator requires more computing time, the new estimator gains a higher efficiency in covariance estimation in the case of high-frequency noisy observations.\\
We present the results derived in this section again in the following proposition which implies Theorem \ref{multitheo}.
\begin{prop} \label{multi}If we choose $M_N=\mathcal{O}(\sqrt{N})$ and calculate the noise-optimal multi-scale estimator for the integrated covariance 
\begin{align}\widehat{\langle \tilde X,\tilde Y\rangle}_T^{mult}=\sum_{i=1}^{M_N}\left(\frac{12i}{M_N^3}-\frac{6}{M_N^2}\right)\sum_{j=i}^N\left(X_{g_j}-X_{l_{j-i}}\right)\left(Y_{\gamma_j}-Y_{\lambda_{j-i}}\right)~,\end{align}
we obtain a consistent unbiased estimatior with asymtotic variance of order $M^{-1}=N^{-\nicefrac{1}{2}}$:
\begin{subequations}
\begin{align}\label{ew}\E\left[\widehat{\langle \tilde X,\tilde Y\rangle}^{mult}_T-\langle \tilde X,\tilde Y\rangle_T\right]=0\end{align}
\begin{align}\var\left(\widehat{\langle \tilde X,\tilde Y\rangle}^{mult}_T\right)=\mathcal{O}\left(\frac{1}{M}\right)=\mathcal{O}\left(
\sqrt{\frac{1}{N}}\right)~.\end{align}
The bias-variance decomposition yields that
\begin{align}\label{bvz}\left(\E\left[\left(\widehat{\langle \tilde X,\tilde Y\rangle}_T^{mult}-\langle \tilde X,\tilde Y\rangle_T\right)^2\right]\right)^{\nicefrac{1}{2}}=\mathcal{O}\left(N^{-\nicefrac{1}{4}}\right)~.\end{align}
\end{subequations}
We will prove the rate of convergence $N^{\nicefrac{1}{4}}$ to be optimal in the following section.
\end{prop}
\begin{remark}We suppose that an extension of Proposition \ref{multi} for non-i.\,i.\,d.\,noise is possible under a milder assumption of exponentially decreasing mixing coefficients such that the equations \eqref{ew}-\eqref{bvz} still hold. This extension for the one-dimensional case has been developed in \cite{zhangmykland2}.\end{remark}

\section{A lower bound for the rate of convergence\label{sec:6}}
In the following we show the LAN (local asymptotic normality) property for a constant correlation coefficient $\rho=corr(B^X,B^Y)$ of the two Brownian motions of $\tilde X$ and $\tilde Y$ with rate $N^{-\nicefrac{1}{4}}$ within the following simplified model and conclude the rate-optimality of our estimator defined in the last chapter.\\
We have the observations:
\begin{align*} X_{t_i}&=\int_0^{t_i} dB_t^X+\epsilon_{t_i}^X\\
							 Y_{t_i}&=\int_0^{t_i} dB_t^Y+\epsilon_{t_i}^Y~~~~i=0,\ldots,N~.\end{align*}

We restrict ourselves to synchronous observations and equidistant time intervals $\d t_i=\d t= 1/N$. Furthermore we assume the discrete noise processes to be independent of the efficient processes and independent to each other (as in Assumption \ref{e} before). We strengthen the i.i.d.\,assumption for the noise to an i.i.d.\,-Gaussian assumption:
$$\epsilon_{t_i}^X \stackrel{iid}{\sim}\mathcal{N}(0,\eta_X^2)~,~\epsilon_{t_i}^Y\stackrel{iid}{\sim}\mathcal{N}(0,\eta_Y^2)~,i=0,\ldots,N~.$$\\
We want to estimate the parameter $\rho$ from observed increments $\left(\d X_{t_1},\ldots,\d X_{t_N},\d Y_{t_1},\ldots,\d Y_{t_N}\right)$ taking values in a measurable space $\left(\Omega_{2N},\mathcal{F}_{2N}\right)$ with law $\P_{\rho}^{2N}$.
Local asymptotic normality with rate $N^{-\nicefrac{1}{4}}$ means that for a real sequence $h_N\rightarrow h$ the sequence of log-likelihoods converges in law to a limit of the following form:
$$\log\left(\frac{d\P_{\rho+N^{-\frac{1}{4}}h_N}^{2N}}{d\P_{\rho}^{2N}}\right)\stackrel{\P_{\rho}^{2N}}{\longrightarrow}hZ\sqrt{I\left(\rho\right)}-\frac{h^2I\left(\rho\right)}{2}$$
with $Z \sim \mathcal{N}(0,1)$ and $I(\rho)$ denoting the Fisher information.
Then the limit distribution of a sequence of estimators $\hat \rho_{2N}$ is under regularity conditions the convolution of a Gaussian distribution and a noise factor. The maximum risk of any estimator is bounded below by the Gaussian risk and the minimax theorem gives the result on how well the parameter can be estimated asymptotically. See e.\,g.\, \cite{vandervaart} and \cite{vandervaart2} for further information on LAN and optimal convergence rates.\\
We summarize the results of this section in the following Proposition \ref{bound}:
\begin{prop}\label{bound}In the simple model of two synchronously equidistantly observed standard Brownian motions $\tilde X$ and $\tilde Y$ with constant correlation $\rho$ and an observation noise described by i.i.d.\,Gaussian errors with standard deviations $\eta_X$ and $\eta_Y$ the LAN property with $N^{-\nicefrac{1}{4}}$-rate holds , where $N$ denotes the number of observations in the interval $[0,1]$. Assuming without loss of generality $\eta_X\ge \eta_Y$, we obtain the following lower and upper bound for the asymptotic Fisher information:
\begin{equation}\label{fisherbounds}\frac{1}{8\eta_X}\left(\frac{1}{(1+\rho)^{\nicefrac{3}{2}}}+\frac{1}{(1-\rho)^{\nicefrac{3}{2}}}\right)\le I(\rho)\le \frac{\sqrt{2}}{8}\frac{1}{\sqrt{\eta_X^2+\eta_Y^2}}\left(\frac{1}{(1+\rho)^{\nicefrac{3}{2}}}+\frac{1}{(1-\rho)^{\nicefrac{3}{2}}}\right)~.\end{equation}
Particularly assuming the variance of both noise processes to be equal $\left(\eta_X=\eta_Y=\eta\right)$ we can calculate the exact asymptotic Fisher information.
It is given by \begin{equation}\label{fisher}I(\rho)=\frac{1}{8\eta}\left(\frac{1}{(1+\rho)^{\nicefrac{3}{2}}}+\frac{1}{(1-\rho)^{\nicefrac{3}{2}}}\right)~.\end{equation}.\end{prop}
Proposition \ref{bound} implies Theroem \ref{boundtheo} and gives, furthermore, bounds for the asymptotic Fisher information. 
\begin{remark}
We prove the LAN property with rate $N^{-\nicefrac{1}{4}}$ in this simplified model and thus the optimality of our multi-scale estimator. It has the optimal rate of convergence even in the synchronous equidistant case. The asymptotic Fisher information is enclosed between the `natural' lower and an intuitive upper bound. We state that the Fisher information \eqref{fisher} has the following asymptotic behaviour:
$$I\left(\rho\right)\rightarrow \infty ~~\mbox{for} ~~\rho\rightarrow \pm 1~~\mbox{and}~~I\left(\rho\right)\rightarrow 0~~\mbox{for}~~\eta\rightarrow \infty~.$$
\end{remark}
\begin{proof}
First we will prove the LAN property for the simpler case of equal noise variances $\eta_X=\eta_Y=\eta$ and calculate the asymptotic Fisher information \eqref{fisher}.
We want to derive the distribution of the increments
\begin{align*}~~~~~~~~~\d X_{t_i}&=\int_{t_{i-1}}^{t_i}dB_t^X+\epsilon_{t_i}^X-\epsilon_{t_{i-1}}^X~,\\
					\mbox{and}~~~	~	\d Y_{t_i}&=\int_{t_{i-1}}^{t_i}dB_t^Y+\epsilon_{t_i}^Y-\epsilon_{t_{i-1}}^Y~.\end{align*}	
The constant correlation parameter is denoted by $\theta$ in the following.												
There exists a Brownian motion $B$ independent of $\tilde X$ such that the following equation holds:
$$\d Y_{t_i}=\int_{t_{i-1}}^{t_i}\theta dB_t^X+\sqrt{1-\theta^2}\int_{t_{i-1}}^{t_i}dB_t+\epsilon_{t_i}^Y-\epsilon_{t_{i-1}}^Y~.$$
Taking this into account we can easily calculate the covariations of the increments:
\begin{align*} \cov(\d X_{t_i},\d X_{t_j})=\cov(\d Y_{t_i},\d Y_{t_j})=\begin{cases} \d t+2\eta^2 ~~\mbox{if}~~~~~~~~~~ i=j\\~~ -\eta^2 ~~~~\,~~\mbox{if}~|i-j|=1\\ ~~~~~~\, 0 ~~~~~~~\,~\mbox{if}~|i-j|>1\end{cases},\end{align*}
\begin{align*}\cov(\d X_{t_i},\d Y_{t_j})=\begin{cases} \theta \d t ~~\mbox{if}~&i=j\\ ~~0 ~~~~~\mbox{if}~ &i\ne j\end{cases}~.\end{align*}
The random vector $(\d X_{t_1},\ldots,\d X_{t_N},\d Y_{t_1},\ldots,\d Y_{t_N})^t$ has a $2N \times 2N$ dimensional covariance matrix
\begin{align*}\Sigma_{\theta}=\left(\begin{array}{cc} A_N & D_N\\ D_N & A_N\end{array}\right)\end{align*}
with the $N \times N$ diagonal matrix
\begin{align*}D_N=\left(\begin{array}{cccc} \theta \d t&0 & \ldots &0 \\ 0& \ddots & & \vdots\\ \vdots&&\ddots &0 \\ 0&\ldots&0&\theta \d t\end{array}\right)\end{align*}
and the $N\times N$ tridiagonal 1-Toeplitz matrix
\begin{align*}A_N=\left(\begin{array}{ccccc} \d t+2\eta^2&-\eta^2 & 0 &\ldots &0\\ -\eta^2& \ddots& \ddots &&\vdots \\ 0&\ddots&\ddots &\ddots&0 \\ \vdots& &\ddots&\ddots&-\eta^2\\ 0 &\ldots &0&-\eta^2&\d t+2\eta^2 \end{array}\right)~.\end{align*}
This special structure of the covariance matrix makes it possible to explicitly compute the eigenvalues of $\Sigma_{\theta}$. Here the fact that we assumed the variances of both noise processes to be equal plays an important role.\\
We write the $N$-dimensional identity matrix as $\1_N$. Then the characteristic polynomial of $\Sigma_{\theta}$ can be written as
$$\det\left(\Sigma_{\theta}-\lambda \1_{2N}\right)=\left(\det\left(A_N-\lambda\1_N\right)\right)^2-\left(\theta \d t\right)^{2N}~.$$
Using a Laplace-expansion, the characteristic polynomials of $A_N$ can be computed by a recursion:
\begin{align*}\det\left(A_N-\lambda\1_N\right)&=\left(\d t+2\eta^2-\lambda\right)\det\left(A_{N-1}-\lambda \1_{N-1}\right)+\left(\eta^2\right)^2\det\left(A_{N-2}-\lambda\1_{N-2}\right)\\
&=\sum_{k=0}^{\lfloor \frac{N}{2}\rfloor}\left(-1\right)^k\binom{N-k}{k}\left(\d t+2\eta^2-\lambda\right)^{N-2k}\left(\eta^2\right)^{2k}~.\end{align*}
The eigenvalues of $A_N$ are $\lambda_{i,N}=\d t+2\eta^2\left(1-\cos{\frac{i\pi}{N+1}}\right)~,~~i=1,\ldots,N$,
and because of the simple structure of $\Sigma_{\theta}$ we can deduce the $2N$ eigenvalues of the covariance matrix directly:
\begin{subequations}
\begin{align} \lambda_{i,N}^{+}(\theta)=\d t (1+\theta)+2\eta^2\left(1-\cos{\frac{i\pi}{N+1}}\right)~,~~i=1,\ldots,N ~,\\
							\lambda_{i,N}^{-}(\theta)=\d t (1-\theta)+2\eta^2\left(1-\cos{\frac{i\pi}{N+1}}\right)~,~~i=1,\ldots,N ~.							
							\end{align}							
With the notation
\begin{equation} \lambda_{j,2N}(\theta)=\begin{cases} \lambda_{i,N}^+~~~\mbox{if}~~j=2i-1~,~~i=1,\ldots,N~,\\ \lambda_{i,N}^-~~~\mbox{if}~~j=2i~~~~~~~~,~~i=1,\ldots,N\end{cases}~,\end{equation}
\end{subequations}
we can write the $2N \times 2N$ diagonal matrix of the eigenvalues as $\Lambda_{\theta}^{2N}$ with $(\Lambda_{\theta}^{2N})_{jj}=\lambda_{j,2N}(\theta)$. $\Sigma_{\theta}$ can be diagonalized by an $2N \times 2N$ orthogonal matrix $P^{2N}$ which is independent of $\theta$. The random vector $P^{2N} \cdot \left(\d X_{t_1},\ldots,\d X_{t_N},\d Y_{t_1},\ldots,\d Y_{t_N}\right)^t$ is centered Gaussian with covariance matrix $\Lambda_{\theta}^{2N}$. We define the $2N$-dimensional random vector $T^{2N}$ by
\begin{align*}\left(T^{2N}\right)_j\:=\frac{1}{\sqrt{\lambda_{j,2N}(\rho)}}\left(P^{2N}\cdot \left(\d X_{t_1},\ldots,\d X_{t_N},\d Y_{t_1},\ldots,\d Y_{t_N}\right)^t\right)_j\\ ~~~~~~~~~~~~~~~~~~~~~~~~~~~~~~~~~~~~~~~~~~~~~\sim ~\mathcal{N}\left(0\,,\,\frac{\lambda_{j,2N}(\theta)}{\lambda_{j,2N}(\rho)}\right)~.\end{align*}
To prove the LAN property we have to examine the $\log$-likelihood
\begin{align*}&~~\log{\left(\frac{d\P_{\rho+N^{-\frac{1}{4}}h_N}^{2N}}{d\P_{\rho}^{2N}}\right)}\\&=\log{\left[\left(\frac{\prod_{j=1}^{2N}\lambda_{j,2N}\left(\rho+N^{-\frac{1}{4}}h_N\right)}{\prod_{j=1}^{2N}\lambda_{j,2N}(\rho)}\right)^{-\frac{1}{2}}\right]}-\frac{1}{2}\sum_{j=1}^{2N}\left(T^{2N}\right)_j^2\left(\frac{\lambda_{j,2N}(\rho)}{\lambda_{j,2N}\left(\rho+N^{-\frac{1}{4}}h_N\right)}-1\right)
\\&=-\frac{1}{2}\sum_{j=1}^{2N}\left(\log{\left(1+\gamma_j^{2N}\right)}-\left(T^{2N}\right)_j^2\frac{\gamma_j^{2N}}{\gamma_j^{2N}+1}\right)\end{align*}
where
$$\gamma_j^{2N}\:=\frac{\lambda_{j,2N}\left(\rho+N^{-\nicefrac{1}{4}}h_N\right)}{\lambda_{j,2N}(\rho)}-1=\frac{\d t\cdot N^{-\nicefrac{1}{4}}h_N}{\lambda_{j,2N}(\rho)}~.$$
The proof is now analogous to the one dimensional case (see \cite{gloter}) and using Theorem VIII-3.32 in \cite{shir} it remains to show that
\begin{equation} \sup_{1\le j\le 2N}{|\gamma_j^{2N}|}\rightarrow 0~~~\mbox{and} ~~~\sum_{j=1}^{2N}\left(\gamma_j^{2N}\right)^2\rightarrow 2h^2 I(\rho)~~.\end{equation}
The first condition is obviously fulfilled. To prove the second one we write the sum of the squares as a Riemann sum and use an inequality including the corresponding integral:
\begin{align*}
\sum_{j=1}^{2N}\left(\gamma_j^{2N}\right)^2&=\sum_{j=1}^N\frac{N^{-\nicefrac{1}{2}}h_N^2}{\left(1+\rho+\frac{2\eta^2}{\d t}\left(1-\cos{\frac{j\pi}{N+1}}\right)\right)^2}+\sum_{j=1}^N\frac{N^{-\nicefrac{1}{2}}h_N^2}{\left(1-\rho+\frac{2\eta^2}{\d t}\left(1-\cos{\frac{j\pi}{N+1}}\right)\right)^2}\\ &=\frac{N^{\nicefrac{1}{2}}h_N^2\left(\d t\right)^2}{\left(\eta^2\right)^2\pi} \underbrace{\frac{\pi}{N}\sum_{j=1}^N\frac{1}{\left(2\left(1-\cos{\frac{j\pi}{N+1}}\right)+\frac{\d t(1+\rho)}{\eta^2}\right)^2}}_{=S_N}\\ &~~~~~~~~~~~~+\frac{N^{\nicefrac{1}{2}}h_N^2\left(\d t\right)^2}{\left(\eta^2\right)^2\pi} \underbrace{\frac{\pi}{N}\sum_{j=1}^N\frac{1}{\left(2\left(1-\cos{\frac{j\pi}{N+1}}\right)+\frac{\d t(1-\rho)}{\eta^2}\right)^2}}_{=\tilde S_N}~.
\end{align*}
For the integral
$$J=\int_0^{\pi}\frac{1}{\left(2\left(1-\cos{z}\right)+\frac{\d t(1+\rho)}{\eta^2}\right)^2}dz$$
and accordingly
$$\tilde J=\int_0^{\pi}\frac{1}{\left(2\left(1-\cos{z}\right)+\frac{\d t(1-\rho)}{\eta^2}\right)^2}dz~,$$ the following inequalities with the lower and upper Darboux sums hold:
$$\frac{\pi}{N}\sum_{j=1}^N\frac{1}{\left(2\left(1-\cos{\frac{(j+1)\pi}{N+1}}\right)+\frac{\d t(1+\rho)}{\eta^2}\right)^2}\le J\le \frac{\pi}{N}\sum_{j=1}^N\frac{1}{\left(2\left(1-\cos{\frac{j\pi}{N+1}}\right)+\frac{\d t(1+\rho)}{\eta^2}\right)^2}$$
and accordingly
$$\frac{\pi}{N}\sum_{j=1}^N\frac{1}{\left(2\left(1-\cos{\frac{(j+1)\pi}{N+1}}\right)+\frac{\d t(1-\rho)}{\eta^2}\right)^2}\le \tilde J\le \frac{\pi}{N}\sum_{j=1}^N\frac{1}{\left(2\left(1-\cos{\frac{j\pi}{N+1}}\right)+\frac{\d t(1-\rho)}{\eta^2}\right)^2}~.$$
Thus the following inequalities hold for the Riemann sums $S_N$ and $\tilde S_N$, respectively:
$$J\le S_N\le J+\frac{\pi}{N}\frac{1}{N\left(4+\frac{\d t(1+\rho)}{\eta^2}\right)^2}-\frac{\pi}{N}\frac{1}{\left(2\left(1-\cos{\frac{\pi}{N+1}}\right)+\frac{\d t+(1+\rho)}{\eta^2}\right)^2}$$
and
$$\tilde J\le \tilde S_N\le \tilde J+\frac{\pi}{N}\frac{1}{N\left(4+\frac{\d t(1-\rho)}{\eta^2}\right)^2}-\frac{\pi}{N}\frac{1}{\left(2\left(1-\cos{\frac{\pi}{N+1}}\right)+\frac{\d t+(1-\rho)}{\eta^2}\right)^2}~.$$
The integrals can be computed explicitly:
\begin{align*}\frac{N^{-\nicefrac{3}{2}}h_N^2}{\left(\eta^2\right)^2\pi}\left(J+\tilde J\right)=\frac{N^{-\nicefrac{3}{2}}h_N^2}{\left(\eta^2\right)^2\pi}\left[\frac{\pi\left(2+\frac{\d t(1+\rho)}{\eta^2}\right)}{\left(\frac{\d t(1+\rho)}{\eta^2}\left(\frac{\d t(1+\rho)}{\eta^2}+4\right)\right)^{\nicefrac{3}{2}}}+\frac{\pi\left(2+\frac{\d t(1-\rho)}{\eta^2}\right)}{\left(\frac{\d t(1-\rho)}{\eta^2}\left(\frac{\d t(1-\rho)}{\eta^2}+4\right)\right)^{\nicefrac{3}{2}}}\right]\,.\end{align*}
Since $h_N\rightarrow h$, we can deduce from the preceding inequalities for both summands the convergence
\begin{equation}\sum_{j=1}^{2N}\left(\gamma_j^{2N}\right)^2\rightarrow \frac{h^2}{4\eta}\left(\frac{1}{(1+\rho)^{\nicefrac{3}{2}}}+\frac{1}{(1-\rho)^{\nicefrac{3}{2}}}\right)=2h^2I(\rho)\end{equation}
with the Fisher information\vspace*{1cm} \begin{equation}I(\rho)=\frac{1}{8\eta}\left(\frac{1}{(1+\rho)^{\nicefrac{3}{2}}}+\frac{1}{(1-\rho)^{\nicefrac{3}{2}}}\right)~.\end{equation}
We continue the proof with the generalization for different noise variances. 
If the noise variances are not equal $\eta_X^2\ne \eta_Y^2$, the covariance matrix can be written as
\begin{align*}\Sigma_{\theta}=\left(\begin{array}{cc} A_N & D_N\\ D_N & B_N\end{array}\right)\end{align*}
with the same diagonal matrix $D_N$ as before and two tridiagonal 1-Toeplitz matrices $A_N$ and $B_N$ with the same structure as before where $A_N$ has the entries $\d t+2\eta_X^2$ on the main diagonal and correspondingly, $B_N$ the entries $\d t +2\eta_Y^2$. The eigenvalues of $A_N$ and $B_N$ have been deduced before and are denoted by $\lambda_X^{(i,N)}$ and $\lambda_Y^{(i,N)}$ here, which emphasizes the dependence on $\eta_X$ and $\eta_Y$, respectively. Because of the special structure of $A_N$ and $B_N$, that are in particular symmetric and commutative, they share the same eingenvectors $v_i, i=1,\ldots,N$. We can calculate the $2N$ eigenvalues of $\Sigma_{\theta}$, denoted by $\xi_{+}^{(i)},\xi_{-}^{(i)},i=1,\ldots,N$, using the approach
\begin{align*}\Sigma_{\theta}=\left(\begin{array}{cc} A_N & D_N\\ D_N & B_N\end{array}\right)\cdot \left(\begin{array}{c} \alpha v_i\\ \beta v_i\end{array}\right)=\xi\left(\begin{array}{c} \alpha v_i\\ \beta v_i\end{array}\right) \end{align*}
for the eigenvectors with constants $\alpha$ and $\beta$.
This equation implies that
$$\alpha \,\lambda_X^{(i,N)}+\d t\,\theta\,\beta=\alpha \,\xi~,$$
$$\alpha \,\d t \,\theta+\,\beta\, \lambda_Y^{(i,N)}=\beta \,\xi~,$$
and by solving this system of equations we obtain the $2N$ eigenvalues

$$\xi_{+}^{(i)}=\frac{\lambda_X^{(i)}+\lambda_Y^{(i)}}{2}+\sqrt{\left(\frac{\lambda_X^{(i)}-\lambda_Y^{(i)}}{2}\right)^2+\theta^2\left(\d t\right)^2}~,$$
$$\xi_{-}^{(i)}=\frac{\lambda_X^{(i)}+\lambda_Y^{(i)}}{2}-\sqrt{\left(\frac{\lambda_X^{(i)}-\lambda_Y^{(i)}}{2}\right)^2+\theta^2\left(\d t\right)^2}~.$$
\noindent
We have dropped the index $N$ of the eigenvalues here.
\begin{lem}\label{evlem}If we assume $\eta_X>\eta_Y$, the following inequalities hold:
\begin{subequations}
\begin{equation}\label{ineq1}\frac{\lambda_X^{(i)}+\lambda_Y^{(i)}}{2}+\theta \d t<\xi_{+}^{(i)}<\lambda_X^{(i)}+\theta\d t~,\end{equation}
\begin{equation}\label{ineq2}\lambda_X^{(i)}-\theta\d t<\xi_{-}^{(i)}<\frac{\lambda_X^{(i)}+\lambda_Y^{(i)}}{2}-\theta \d t~.\end{equation}
\end{subequations}\end{lem}\enlargethispage*{2cm}
\begin{proof}
If $\eta_X>\eta_Y$ for the eigenvalues $\lambda_X^{(i)}>\lambda_Y^{(i)}$ holds for all $i \in \{1,\ldots,N\}$. Thus
$$\xi_{+}^{(i)}<\frac{\lambda_X^{(i)}+\lambda_Y^{(i)}}{2}+\sqrt{\left(\frac{\lambda_X^{(i)}-\lambda_Y^{(i)}}{2}\right)^2+\left(\lambda_X^{(i)}-\lambda_Y^{(i)}\right)\theta \d t+\theta^2\left(\d t\right)^2}=\lambda_X^{(i)}+\theta \d t$$
holds and analogously the lower bound for $\xi_{-}^{(i)}$ is obtained by adding the mixed term to the expression under the square root.
The other bounds are obvious.
\end{proof}
In the following we define
$$\gamma_{+}^{(i)}=\frac{\xi_{+}^{(i)}\left(\rho+N^{-\nicefrac{1}{4}}h_N\right)}{\xi_{+}^{(i)}\left(\rho\right)}-1>0~~\mbox{and}$$
$$\gamma_{-}^{(i)}=\frac{\xi_{-}^{(i)}\left(\rho+N^{-\nicefrac{1}{4}}h_N\right)}{\xi_{-}^{(i)}\left(\rho\right)}-1<0$$
in analogy to the case of equal noise variances. We use the preceding lemma to obtain bounds for these coefficients and show the LAN property with the same rate $N^{-\nicefrac{1}{4}}$ as above, including bounds for the Fisher information.
\begin{prop}\label{propineq}
If $\eta_X>\eta_Y$ the following inequalities hold:
\begin{subequations}
\begin{equation}\frac{N^{-\frac{1}{4}}h_N\d t+\frac{\lambda_Y^{(i)}-\lambda_X^{(i)}}{2}}{\lambda_X^{(i)}+\rho\d t}<\gamma_{+}^{(i)}<\frac{N^{-\frac{1}{4}}h_N\d t}{\frac{\lambda_X^{(i)}+\lambda_Y^{(i)}}{2}+\rho\d t}\end{equation}
and
\begin{equation}\hspace*{2.8cm}\frac{-N^{-\frac{1}{4}}h_N\d t}{\frac{\lambda_X^{(i)}+\lambda_Y^{(i)}}{2}-\rho\d t}<\gamma_{-}^{(i)}<\frac{-N^{-\frac{1}{4}}h_N\d t+\frac{\lambda_Y^{(i)}-\lambda_X^{(i)}}{2}}{\lambda_X^{(i)}-\rho\d t}~.\end{equation}\end{subequations}\end{prop}
\begin{proof}
Using the inequality \eqref{ineq1} in the preceding Lemma \ref{evlem} we obtain the lower bound for $\gamma_{+}^{(i)}$. From 
\begin{align*}&\gamma_{+}^{(i)}=\frac{\frac{\lambda_X^{(i)}+\lambda_Y^{(i)}}{2}+\sqrt{\left(\frac{\lambda_X^{(i)}-\lambda_Y^{(i)}}{2}\right)^2+\left(\rho+N^{-\frac{1}{4}}h_N\right)^2\left(\d t\right)^2}}{\xi_{+}^{(i)}(\rho)}-1\\
<&\frac{\frac{\lambda_X^{(i)}+\lambda_Y^{(i)}}{2}\hspace*{-0.05cm}+\hspace*{-0.07cm}\small \sqrt{\hspace*{-0.07cm}\left(\frac{\lambda_X^{(i)}-\lambda_Y^{(i)}}{2}\right)^2\hspace*{-0.125cm}+\hspace*{-0.035cm}\rho^2\hspace*{-0.01cm}\left(\d t\right)^2\hspace*{-0.06cm}+\hspace*{-0.015cm}N^{-\frac{1}{2}}h_N^2\hspace*{-0.015cm}\left(\d t\right)^2\hspace*{-0.075cm}+\hspace*{-0.025cm}2N^{-\frac{1}{4}}h_N\d t\sqrt{\hspace*{-0.015cm}\left(\frac{\lambda_X^{(i)}-\lambda_Y^{(i)}}{2}\right)^2\hspace*{-0.125cm}+\hspace*{-0.01cm}\rho^2\hspace*{-0.01cm}\left(\d t\right)^2}}}{\xi_{+}^{(i)}(\rho)}\normalsize-1\\
&=\frac{\frac{\lambda_X^{(i)}+\lambda_Y^{(i)}}{2}+\sqrt{\left(\frac{\lambda_X^{(i)}-\lambda_Y^{(i)}}{2}\right)^2+\left(\rho \d t\right)^2}+N^{-\frac{1}{4}}h_N\d t}{\xi_{+}^{(i)}(\rho)}-1\\
&=\frac{N^{-\frac{1}{4}}h_N\d t}{\xi_{+}^{(i)}(\rho)}<\frac{N^{-\frac{1}{4}}h_N\d t}{\frac{\lambda_X^{(i)}+\lambda_Y^{(i)}}{2}+\rho \d t}\end{align*}
we can deduce the upper bound using again the right-hand side of inequality \eqref{ineq1} in the last inequality. The bounds for $\gamma_{-}^{(i)}$ follow analogously.
\end{proof}
Now we are able to prove the LAN property in the same way as for the case of equal noise variances using the preceding inequalities.
Because of Proposition \ref{propineq}, the inequalities
$$\sum_{i=1}^N\left(\gamma_{+}^{(i)}\right)^2+\sum_{i=1}^N\left(\gamma_{-}^{(i)}\right)^2<\sum_{i=1}^N\left(\frac{N^{-\frac{1}{2}}h_N^2\left(\d t\right)^2}{\left(\frac{\lambda_X^{(i)}+\lambda_Y^{(i)}}{2}+\rho\d t\right)^2}+\frac{N^{-\frac{1}{2}}h_n^2\left(\d t\right)^2}{\left(\frac{\lambda_X^{(i)}+\lambda_Y^{(i)}}{2}-\rho\d t\right)^2}\right)$$
and
\begin{align*}\sum_{i=1}^N\left(\gamma_{+}^{(i)}\right)^2\hspace*{-0.1cm}+\hspace*{-0.1cm}\sum_{i=1}^N\left(\gamma_{-}^{(i)}\right)^2&>\sum_{i=1}^N\hspace*{-0.1cm}\left(\hspace*{-0.1cm}\frac{N^{-\frac{1}{2}}h_N^2\left(\d t\right)^2\hspace*{-0.05cm}+\hspace*{-0.05cm}\left(\frac{\lambda_Y^{(i)}-\lambda_X^{(i)}}{2}\right)^2}{\left(\lambda_X^{(i)}+\rho\d t\right)^2}\hspace*{-0.02cm}+\hspace*{-0.02cm}\frac{N^{-\frac{1}{2}}h_N^2\left(\d t\right)^2\hspace*{-0.05cm}+\hspace*{-0.05cm}\left(\frac{\lambda_Y^{(i)}-\lambda_X^{(i)}}{2}\right)^2}{\left(\lambda_X^{(i)}-\rho\d t\right)^2}\hspace*{-0.01cm}\right)\\ &>\sum_{i=1}^N\left(\frac{N^{-\frac{1}{2}}h_N^2\left(\d t\right)^2}{\left(\lambda_X^{(i)}+\rho\d t\right)^2}+\frac{N^{-\frac{1}{2}}h_N^2\left(\d t\right)^2}{\left(\lambda_X^{(i)}-\rho\d t\right)^2}\right)\end{align*}
hold. In the lower bound the mixed terms drop out.\\
Using those inequalities, the proof reduces to the method used before for the equal noise variance case where we found that (Riemann) sums of this type can be approximated by integrals. We just have to do this calculation twice for the upper and the lower bound changing only the constants in the denominator of the integrated function and obtain the convergence to $2h^2\underline{I(\rho)}$ and $2h^2\overline{I(\rho)}$, respectively, with the lower and upper bound for $I\left(\rho\right)$ stated in \eqref{fisherbounds}. 
\end{proof}
\begin{remark}Although the inequalities appearing in the proof for the case of different noise variances are strict, the asymptotic results do not yield the strict inequalities in \eqref{fisherbounds} for the lower and upper bound for the asymptotic Fisher information. We suppose that the strict inequalities also hold and a numerical approximation for the Riemann sums using different special values indicated this too. \end{remark}

\section{Simulation Results\label{sec:7}}
In this section we compare the simulation results for the subsampling and the multi-scale estimator. A detailed comparison of the finite sample size performance of the subsampling and the Hayashi-Yoshida estimator is given in \cite{palandri}. We have shown that the multi-scale estimator is asymptotically more efficient which means that the rate of convergence is $N^{\nicefrac{1}{4}}$ compared to $N^{\nicefrac{1}{6}}$ for the subsampling estimator. The following simulations investigate the behaviour of both estimators for finite sample sizes.\\
To generate asynchronous observation times for the processes $X$ and $Y$ we take $ t_i, i=1,\ldots,n$ and $ \tau_j,j=1,\ldots,m$ as arrival times of two independent Poisson processes such that $\d t_i \sim \text{Exp}(\vartheta_X)$ and $\d \tau_j\sim \text{Exp}(\vartheta_Y)$. In our notation this means $\E\left[\d t_i\right]=\vartheta_X$ and $\E\left[\d \tau_j\right]=\vartheta_Y$.
We set $T=1$ and the means of the time increments between observations equally to $\vartheta_X=\vartheta_Y=1/30000$.  The expected number of observations for both processes is about the number of seconds during one trading day and thus a typically high-frequency observation scheme.  The sets of observations $\mathcal{O}^X$ and $\mathcal{O}^Y$ almost surely have no intersection points, but all time increments are of order $1/N$ in probability (and thus all assumptions imposed in the sections before are guaranteed).\\
\begin{remark}In this special case, where the number of observations $n$ and $m$ follow independent Poisson distributions with parameters $1/\vartheta_X$ and $1/\vartheta_Y$, we can prove that our synchronization method creates $N$ synchronized observations with $\E N=1/\vartheta$
where
\begin{equation*}\vartheta=\vartheta_X+\vartheta_Y-\frac{\vartheta_X\vartheta_Y}{\vartheta_X+\vartheta_Y}~.\end{equation*}
For $\vartheta_X=\vartheta_Y(=1/30000)$ we obtain $\vartheta=(3/2)\vartheta_X$ and $\E N=(2/3)\vartheta_X^{-1}(=20000)$.
\end{remark}
For our simulations we use constant parameters $\sigma_X=\sigma_Y=1$ and $\rho \in [-1,1]$ and neglect drift terms. The increments of the efficient processes are then given by
$$\d \tilde X_{t_i}=\int_{t_{i-1}}^{t_i}dB_t^X$$
and
$$\d \tilde Y_{\tau_j}=\int_{\tau_{j-1}}^{\tau_j} dB_t^Y=\int_{\tau_{j-1}}^{\tau_j}\rho dB_t^X+\sqrt{1-\rho^2}\int_{\tau_{j-1}}^{\tau_j}dB_t$$
where $B_t$ is a standard Brownian motion independent of $\tilde X$. Therefore, we simulate values of $\tilde X$ for all observation times in $\mathcal{O}^X\cup\mathcal{O}^Y$ and simulate the observations of $\tilde Y$ using the equation above.
For the dicrete noise processes we assume $\epsilon_{t_i}^X\sim\mathcal{N}(0,\eta_X^2)$ and $\epsilon_{\tau_j}^Y\sim\mathcal{N}(0,\eta_Y^2)$.\\
To calculate the subsampling and the multi-scale estimators we first have to determine the number of subsamples $K_N$ and the number of frequencies $M_N$, respectively. We know that a choice $K_N=c_{sub}N^{\nicefrac{2}{3}}$ and $M_N=c_{multi}N^{\nicefrac{1}{2}}$, respectively, with constants $c_{sub}$ and $c_{multi}$, respectively, minimizes the resulting mean square errors of the estimators. For our simulations we can calculate the optimal constants because we know all the parameters. Inserting $K_N=c_{sub}N^{\nicefrac{2}{3}}$ and $M_N=c_{multi}N^{\nicefrac{1}{2}}$, respectively, in the asymptotic variances of the estimators, minimization yields
$$c_{sub,opt}=\sqrt[3]{3\eta_X^2\eta_Y^2}~,~c_{multi,opt}=\sqrt[4]{\frac{36\cdot 35}{52}\eta_x^2\eta_Y^2}~.$$


If one uses the estimators for data of two asset prices the parameters are unknown and one can calculate the optimal constants above in the same way, but they will depend on the variances of the noise processes and the quarticities of both underlying It\^{o} processes. Then one can estimate the constants by using adequate estimators for these parameters as proposed in \cite{zhang} for example.

\begin{figure}
\centering
\parbox{2.5in}{\begin{tabular}{|c|c|c|}\hline
noise level $\eta_X^2=\eta_Y^2$ & $K_N$ & $M_N$\\
\hline 
$(1/\sqrt{10})$ & 646 & 216\\
0.1 & 300 & 122\\
$(1/\sqrt{10})\cdot$ 0.1& 139 & 68\\
0.01 & 65 & 38\\
$(1/\sqrt{10})\cdot$ 0.01 & 30 & 22\\
0.001 & 14 & 12\\
$(1/\sqrt{10})\cdot$ 0.001 & 6 & 7\\
0.0001 & 3 & 4\\
\hline
\end{tabular}}%
\qquad
\begin{minipage}{2.75in}%
\includegraphics[width=6.5cm]{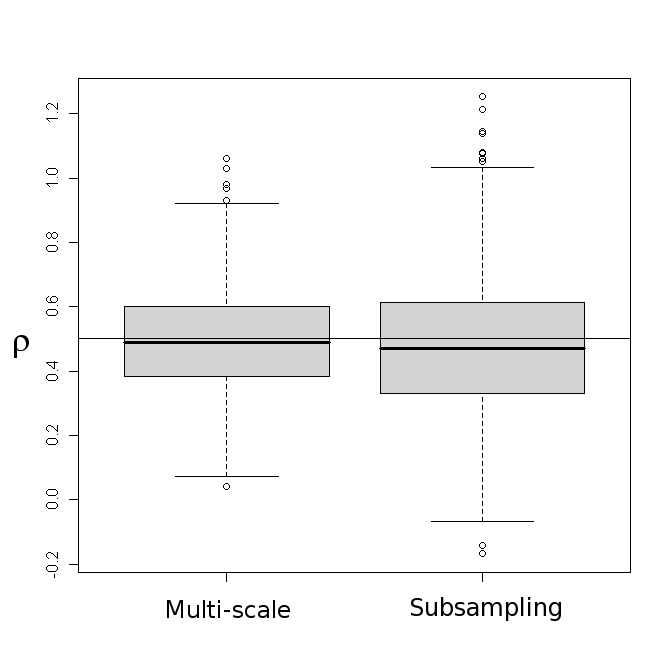}
\end{minipage}
\renewcommand{\figurename}{Table}
\caption{ Calculated values for $K_N$ and $M_N$ for different noise levels $\eta_X^2=\eta_Y^2$ and parameters $\sigma^X=\sigma^Y=1$.}
\renewcommand{\figurename}{Figure}
\caption{\label{bp}Boxplot for Multi-scale and subsampling estimator for $\rho=0.5$ when $\eta_X^2=\eta_Y^2=\eta^2=\sqrt{0.1}$.}
\end{figure}

Figure \ref{bp} shows a boxplot for 1000 Monte Carlo iterations for large noise variances $\eta_X^2=\eta_Y^2=\eta^2=\sqrt{0.1}$ that exemplifies a higher efficiency of our proposed multi-scale estimator compared to the subsampling estimator at least when microstructure noise effects are large.\\
Next we present a comparison of the resulting root mean square errors (RMSE) for different noise levels. The results are illustrated in Figure \ref{rmse}. The RMSEs are calculated for each noise level based on 1000 Monte Carlo iterations.
Our simulations show that for very noisy data (noise level $\eta_X^2=\eta_Y^2=\eta^2\ge 0.01$) and $\E N=30000$ expected observations for both processes, the multi-scale estimator has a significant smaller root mean square error compared to the subsampling estimator. The ratio of both RMSEs is increasing when the noise level decreases in the range $0.1\ge \eta^2$. For small noise levels and same (expected) sample sizes the multi-scale estimator also has a smaller RMSE but the ratio of the RMSEs gets close to 1 and fluctuates for different (small) noise levels.\\ Our simulations thus confirm that for not negligible market microstructure frictions our proposed multi-scale estimator for the quadratic covariation of two It\^{o} processes performs better than the subsampling (and of course the HY-estimator) not only asymptotically but also in the case of typical sample sizes (for high-frequent intraday stock data).\\  \\

\begin{figure}[t]
\begin{center}
\includegraphics[width=14.5cm]{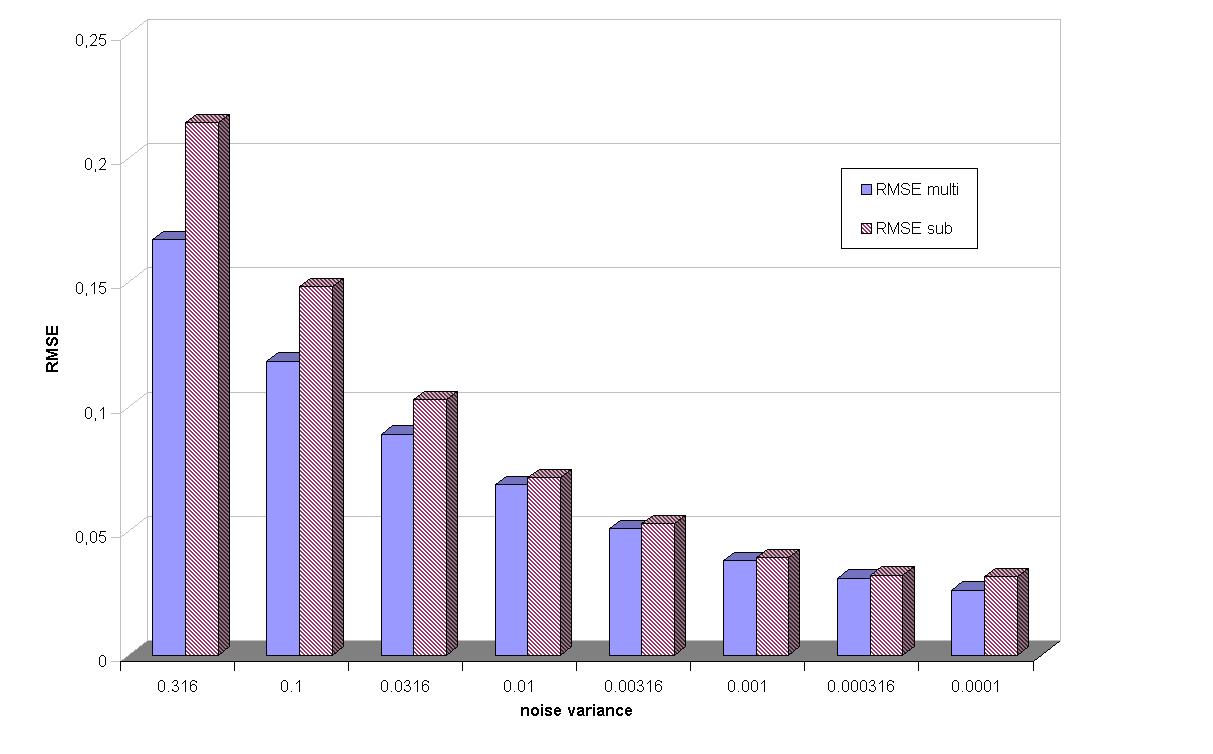}
\caption{\label{rmse}Root mean square errors of subsampling and multi-scale estimator for different noise levels $\eta_X^2=\eta_Y^2=\eta^2$ for $\rho=0.5$.}
\end{center}
\end{figure}

We have chosen the ranges for the noise variances such that for the illustrated values the noise variances decrease with a factor $1/\sqrt{10}$. Because of the factor $\eta_X^2 \eta_Y^2$ in the variances of the estimators due to noise the root mean square error (disregarding the discretization error) should decrease linearly. This can be seen regarding the values in Figure \ref{rmse} for large noise variances when the error due to noise dominates the error due to discretization whereas the influence of the discretization error is stronger for small noise levels.\\
\begin{figure}[t]
\begin{center}
\includegraphics[width=14.5cm]{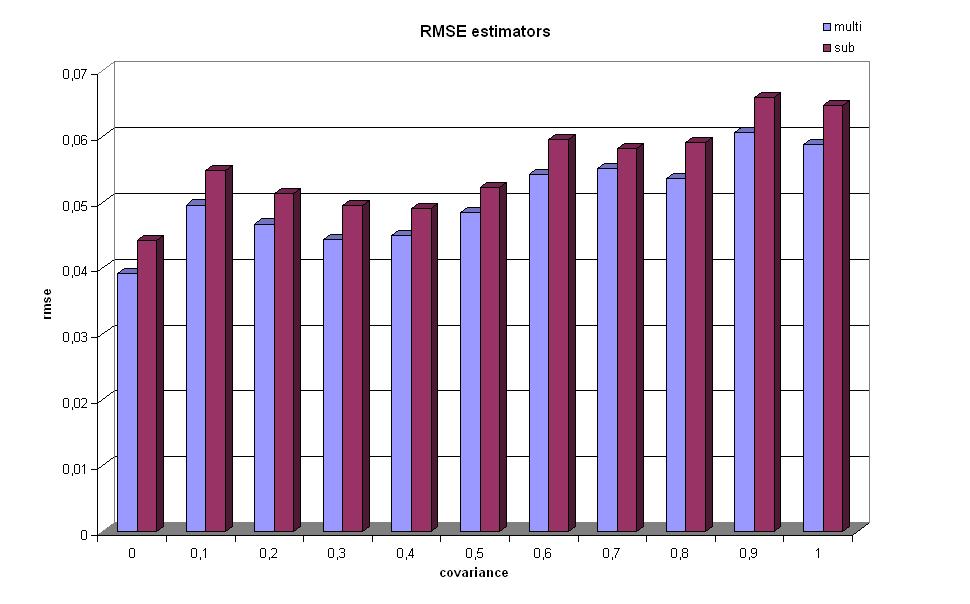}
\caption{\label{ratio}Root mean square errors of subsampling and multi-scale estimator for a constant noise level $\eta_X^2=\eta_Y^2=\eta^2=0.01$ and different correlations.}
\end{center}
\end{figure}

In Figure \ref{ratio} the root mean square errors of both estimators are diagrammed for different constant parameter values of the correlation  $\rho=k/10\,,k=0,\ldots,10$ when $\eta=0.01$ based on 200 Monte Carlo iterations for each value. For all eleven parameter values the multi-scale estimator has a smaller root mean square error although the differences are not very large for this (small) noise level. We can announce the increasing root mean square errors when $\rho$ increases with the dependence of the discretization error on $\rho$. Although we did not state precisely a formula for the asymptotic variance due to discretization it is natural that the variance analyzed in Section \ref{sec:5} grows for higher values of $\rho$. For this noise level the discretization error is influential enough to cause the different root mean square errors illustrated in Figure \ref{ratio}.

\newpage\noindent

\section{Conclusion\label{sec:9}}
We have presented and compared three estimators for the quadratic covariation of two It\^{o} processes. If we have discrete asynchronous observations without market microstructure noise, the Hayashi-Yoshida estimator is a consistent estimator solving the problem of asynchronicity. However, when microstructure frictions are relevant the estimator is not consistent any more, as we stated in Proposition \ref{propdiv}. If we deal with noisy asynchronous data, we have to synchronize the observations first. We used the method presented by \cite{palandri} in Section \ref{sec:3} to rearrange the observations in an adequate way. We have shown in Section \ref{sec:4} that a subsampling approach yields a consistent estimator with $N^{\nicefrac{1}{6}}$-rate of convergence, where $N$ denotes the number of synchronized observations. In Section \ref{sec:5} we introduced our multi-scale estimator which gains a higher efficiency and has a $N^{\nicefrac{1}{4}}$-rate of convergence which is stated in Theorem \ref{multitheo}. This rate is optimal what we have proved in a simplified model even for the synchronous case by giving a lower bound for the rate of convergence using the LAN property with rate $N^{-\nicefrac{1}{4}}$ in Section \ref{sec:6}. Proposition \ref{bound} comprises this result and the asymptotic Fisher information. Simulations show that the multi-scale estimator performs better compared to the subsampling estimator if the noise level is high enough and the sample size is not too small.
\nocite{bandi}
\nocite{poldi}
\bibliographystyle{chicago}
\bibliography{literatur}


\end{document}